\documentclass[aos,preprint]{imsart}
\setattribute{journal}{name}{}

\RequirePackage[numbers]{natbib}
\RequirePackage[colorlinks,citecolor=blue,urlcolor=blue]{hyperref}
\RequirePackage{hypernat}
\usepackage{amssymb}
\usepackage{amsmath,amsthm}
\usepackage{algorithm}
\usepackage[noend]{algpseudocode}
\usepackage{enumitem}
\usepackage[bottom]{footmisc}

\usepackage{latexsym}
\usepackage{epsfig}

\newcommand{\sign}{\mathrm{sign}}

\newcommand{\EE}{{\mathbb{E}}}

\newcommand{\eps}{\varepsilon}

\newcommand{\FF}{\mathcal{F}}

\newcommand{\EEE}{\mathcal{E}}

\renewcommand{\phi}{\varphi}

\newcommand{\given}{\,|\,}

\makeatletter
\def\BState{\State\hskip-\ALG@thistlm}
\makeatother

\newtheorem{theorem}{Theorem}[section]

\newtheorem{proposition}[theorem]{Proposition}
\newtheorem{lemma}[theorem]{Lemma}

\newtheorem{corollary}[theorem]{Corollary}
\theoremstyle{definition}

\numberwithin{equation}{section}

\begin{document}

\begin{frontmatter}
\title{Distributed  function estimation: adaptation
           using minimal communication}
\runtitle{Adaptation with minimal communication}

\begin{aug}
\author{\fnms{Botond} \snm{Szab\'o}\thanksref{t1,t2}\ead[label=e1]{b.t.szabo@math.leidenuniv.nl}},
\and
\author{\fnms{Harry} \snm{van Zanten}\thanksref{t1}\ead[label=e2]{j.h.van.zanten@vu.nl}}

\thankstext{t1}{Research supported by the Netherlands Organization of Scientific Research NWO.}
\thankstext{t2}{The research leading to these results has received funding from the European Research Council under ERC Grant Agreement 320637.}
\runauthor{Szab\'o and Van Zanten}

\affiliation{Leiden University and Vrije Universiteit Amsterdam}

\address{Mathematical Institute\\
Leiden University\\
Niels Bohrweg 1\\
2333 CA Leiden\\
The Netherlands\\
\printead{e1}\\
}

\address{Department of Mathematics\\
Vrije Universiteit Amsterdam\\
De Boelelaan 1111 \\
1081 HV Amsterdam\\
The Netherlands\\
\printead{e2}}
\end{aug}

\begin{abstract}
We investigate whether in a distributed setting, adaptive estimation of a smooth function 
at the optimal rate is possible under minimal communication. It turns out that the answer depends
on the risk considered and on the number of servers over which the procedure is distributed. 
We show that for the $L_\infty$-risk, adaptively obtaining optimal rates under
minimal communication is not possible. For the $L_2$-risk, it is possible over a range 
of regularities that depends on the relation between the number of local servers and the total sample size.
\end{abstract}

%\begin{keyword}[class=MSC]
%\kwd[Primary ]{60K35}
%\kwd{60K35}
%\kwd[; secondary ]{60K35}
%\end{keyword}

%\begin{keyword}
%\kwd{sample}
%\kwd{\LaTeXe}
%\end{keyword}

\end{frontmatter}

\section{Introduction}

Distributed methods have attracted a lot of attention in the statistics and machine learning
communities recently. There are several reasons for this, the most prominent ones 
being that they provide a way of dealing with  large datasets and
with privacy considerations. The theoretical literature on distributed methods is still 
rather minimal at the moment. A number of papers have recently investigated fundamental performance
limits in distributed models, in particular pointing out issues that occur in high-dimensional 
or nonparametric problems, see for instance \cite{zhang2013,kleiner:2014,braverman2016communication,rosen2016,jason:2017,battey:2018,szabo:zanten:2017,barnes2019learning, cai2020distributed}. For example, optimal rates 
in distributed function estimation depend on the amount of communication that is allowed, 
and the relation of that amount with the regularity of the unknown function. 
The lower bounds obtained in 
\cite{szabo:zanten:2018} and  \cite{zhu:2018} and the subsequent adaptation results 
in \cite{szabo:zanten:2018} show that in particular, automatically adapting to the smoothness 
of the unknown function is a complicated issue in communication restricted distributed settings.
In the present paper we study this problem from a different, we think relevant and interesting perspective, not restricting 
communication a priori, but asking for rate-optimal procedures that require minimal communication.

In  distributed estimation problems it is of interest to achieve 
high estimation accuracy, while at the same time limiting communication between servers, or cores, 
since this may give rise to undesirable time loss, costs, or congestion. 
In this paper we investigate this problem
for a basic distributed architecture, where we have $m$ local servers 
over which the data is distributed and that each carry out  a statistical procedure using their local data, independently of each other. 
They communicate their result to a central server that performs some aggregation and 
produces a final estimate of the quantity of interest.
The two goals of high accuracy and little communication are conflicting in this setting. 
It is intuitively clear that to achieve high accuracy it is beneficial to have 
a lot of data in one server, which is only possible if the total number of 
local severs $m$ is small, which we will not assume, or the local servers are allowed to communicate 
a lot of information to the central server, which we will consider to be undesirable.

The problem becomes most interesting if the unknown object is high-, or infinite-dimensional.
To be specific, we will consider a distributed signal estimation problem in which the goal is to 
estimate a function $f \in L_2[0,1]$ with (Besov) regularity $s > 0$. (A precise description of the 
model is given in Section \ref{sec: main}.) 
The best accuracy that can be  achieved with respect to the $L_2$-norm can be described by minimax lower bounds.
In the classical, non-distributed setting the minimax lower bound over Besov balls of regularity $s$ 
is known to be of the order $n^{-s/(1+2s)}$, where $n$ is the sample size, or signal-to-noise ratio
(e.g.\ \cite{gine:nickl:2016}).
Recently established lower bounds for distributed nonparametric methods under communication constraints
(see \cite{szabo:zanten:2018}, \cite{zhu:2018}, and Section \ref{sec: main} ahead) show that this optimal rate can also be achieved by distributed 
methods, but only if each local machine is allowed to communicate at least order $n^{1/(1+2s)}$ bits of 
information to the central machine. This is what the authors of \cite{zhu:2018} call the {\em sufficient regime}. 

A distributed strategy that achieves the rate $n^{-s/(1+2s)}$ under the restriction that the local 
machines communicate at most the minimal order $n^{1/(1+2s)}$ bits is easily constructed (see Theorem \ref{theorem: minimaxL2UB}). 
However, this simple strategy uses knowledge of the regularity $s$ of the unknown signal. The real 
interesting question is whether this can be done adaptively, without knowing $s$. This greatly 
complicates the problem, since we do not only want adaptation to smoothness of the estimator, 
but we also require that the local machines determine the maximally allowed number of bits 
in a purely data-driven manner. 

It turns out that whether or not this is possible for the $L_2$-risk depends on the 
relation between the number of machines $m$ and the total sample size, or signal-to-noise ratio $n$. 
We prove that if $m = n^p$ for some $p \in (0,1/2)$, then: 
\begin{itemize}
\item
There exists a distributed estimator that 
is adaptive over any range of regularities $[s_1, s_2]$ such that 
\[
0 < s_1 < s_2 < \frac1{4p} - \frac12,
\]
achieving the optimal rate and transmitting the minimal amount of bits.
\item
If 
\[
s_2 > s_1 > \frac1{4p} - \frac12
\]
however, then there exists no distributed procedure that achieves the optimal rate for 
every signal $f$ with regularity in $\{s_1, s_2\}$, while transmitting the minimal amount of bits.
\end{itemize}
Stated differently, when considering $L_2$-risk, adaptively achieving the optimal rate using minimal communication 
over a range of regularities $[s_1, s_2]$ is possible if and only if 
\[
(2+4s_2) \log m < {\log n}.
\]
This shows that it is problematic if either the number of machines is too large, or 
the range of regularities to which adaptation is required is too large. 

The adaptive, minimal communication procedure that we propose in the first case
implicitly exploits the fact that for the $L_2$-risk, there is a difference between lower bounds
for estimation and testing, see for instance \cite{ingster2,gine:nickl:2016}.
Indeed, we employ the testing result of \cite{carpentier:2015} 
 to extract sufficient information about the regularity of the unknown signal in the local servers, 
 which we then use in the subsequent estimation procedure.
This approach depends crucially on the fact that we consider the $L_2$-risk. 
For the $L_\infty$-risk there is no difference between testing and estimation rates
and this approach breaks down. 
In fact we prove that for the $L_\infty$-norm, 
adaptive estimation at the optimal rate under minimal communication is never possible!

The impossibility results all derive from the fact that in the local servers, sample
size is too small to extract sufficient information about the regularity of a general signal. 
This suggests that if we restrict to a class of ``nice'' signals for which we 
do have access to such smoothness information from limited data, we should be able to obtain optimal
rates and minimal communication adaptively. 
We prove that this is indeed the case if we consider the class of self-similar functions, 
first introduced in \cite{bull:2012} in the context of nonparametric confidence regions, 
where closely related issues occur.  See also for instance 
\cite{cai:low:04,robins:2006,gine:nickl:2010,bull:2013, szabo:etal:2015, rousseau2016asymptotic}.

The remainder of the paper is organized as follows. In the next section we first present the 
minimax lower bounds under communication restrictions that show that if we want to attain
the optimal rate $n^{-s/(1+2s)}$ for estimating $s$-smooth functions in the distributed setting, 
we need to transmit at least order $n^{1/(1+2s)}$ bits from the local machines to the central one. 
For completeness we show that it is easy to obtain the optimal rate under minimal communication 
if $s$ is known. We also prove that if it is assumed that $s$ belong to some known 
range $(s_0, s_{\rm max})$, then adaptation to smoothness over that range is possible 
while transmitting  order $n^{1/(1+2s_0)}$ bits.
After this we present our main results. Theorems \ref{thm: adaptL2negative} and \ref{thm: adaptL2positive} and Corollary 
\ref{thm: adaptL2positive2}  assert that whether simultaneous adaptation over a range of regularities and minimal 
communication is possible for the $L_2$ risk, depends on the relation between the range of regularities
and the number of local machines. Theorem \ref{thm: adaptLinftynegative} shows that simultaneous
adaptation and minimal communication is not possible when $L_\infty$ risk is considered. Finally, 
Theorem \ref{thm: adaptLinftypositive} asserts that it is possible under a self-similarity assumption. 
Proofs and auxiliary results are deferred to Section 3--5 and the appendices.

\subsection{Notations}
For two positive sequences $a_n,b_n$ we use the notation $a_n\lesssim b_n$ if there exists an universal positive constant $C$ such that $a_n\leq C b_n$. Along the lines $a_n\asymp b_n$ denotes that $a_n\lesssim b_n$ and $b_n\lesssim a_n$ hold simultaneously. In the proofs we use the notation $C$ and $c$ for universal constants which value can differ from line to line and denote by $\# S$ or $|S|$ the cardinality of the finite set $S$. Furthermore, let $l(Y)$ denote the length of a binary string $Y$, and $\log x$ denote the logarithm with base $2$, i.e. $\log_2 x$.

\section{Main results}\label{sec: main}
In our analysis we work with the distributed Gaussian white noise model 
also considered for instance in \cite{szabo:zanten:2017}, \cite{zhu:2018}, 
and which can be seen as an idealized version of the nonparametric regression model. 
Our results can in principle be derived in the regression context as well, similar as 
we did in \cite{szabo:zanten:2018}. However, since the additional technical issues would seriously lengthen 
the already long paper and would add no fundamental insight, we formulate everything 
in the signal in white noise setting in this paper.

We assume that we have $m$ machines and in the $i$th machine we observe the random function $X^{(i)}_t$ given by the stochastic differential equation 
\begin{align}
d X^{(i)}_t=f_0(t)dt+\sqrt{\frac{m}{n}}dW_{t}^{(i)},\quad t\in[0,1],\,i=1,2,...,m,\label{model: GWN}
\end{align}
where $W^{(1)},...,W^{(m)}$ are independent standard Wiener processes and $f_0$ is the unknown function of interest. It is common to assume that the unknown true function $f_0$ belongs to some regularity class. We work in our analysis with Besov smoothness classes, more specifically we assume that $f_0\in B_{2,\infty}^s(L)$ or $f_0\in B_{\infty,\infty}^s(L)$, see Appendix \ref{sec: wavelets} for a rigorous introduction of these smoothness classes. The first class is of Sobolev type, while the second one is H\"older type. 

Parallel to each other, the local machines  carry out a local statistical procedure and transmit the results to the central machine, which provides  the final inference about the functional parameter of interest $f_0$ by somehow aggregating the local outcomes. 
There are however constraints on the  communication  between the local and global machines. 
Local machine $i$ is allowed to send at most $B^{(i)}$ bits (on average) to the central machine. 
The central machine will then collect the transmitted bits from the local computers and combine them to a global, aggregated answer. 
More formally, for a target function class $\FF$, we write $\hat{f}\in \mathcal{F}_{\rm dist}(B^{(1)},...,B^{(m)}; \FF)$ if
$\hat f_n$ is a measurable function of messages of length $\hat{B}^{(i)}$ sent from the local machines and
for every $f_0\in\mathcal{F}$ it holds that $E_{f_0}\hat{B}^{(i)}\leq B^{(i)}$ for every $i$.
For simplicity, we will focus on the case $B^{(1)} = \dots = B^{(m)}$ that the communication restriction is the same
for every local machine.

\subsection{Distributed minimax rates}
As a first step we give lower bounds for the minimax risk for the $L_2$-norm. We assume that in each 
local machine we have the same communication budget, i.e. $B^{(1)}=...=B^{(m)}=B$. 
Then the corresponding minimax $L_2$ estimation rates are the following, see also \cite{szabo:zanten:2018,zhu:2018}.

\begin{theorem}\label{thm: minimaxL2LB}
Let $s, L >0$.
\begin{itemize}
\item
If $B\geq n^{1/(1+2s)}/\log m$:
\begin{align*}
\inf_{\hat{f}\in\mathcal{F}_{dist}(B,...,B; B_{2,\infty}^s(L))}\sup_{f_0\in B_{2,\infty}^s(L)} E_{f_0}\|\hat{f}-f_0\|_2^2 \gtrsim n^{-\frac{2s}{1+2s}}.
\end{align*}
\item
If $(n\log (n)/m^{2+2s})^{1/(1+2s)} \leq B \leq  n^{1/(1+2s)}/\log m$:
\begin{align*}
\inf_{\hat{f}\in\mathcal{F}_{dist}(B,...,B; B_{2,\infty}^s(L))}\sup_{f_0\in B_{2,\infty}^s(L)} E_{f_0}\|\hat{f}-f_0\|_2^2 \gtrsim   \Big(\frac{B\log n}{n^{1/(1+2s)}}\Big)^{-\frac{s}{1+s}} n^{-\frac{2s}{1+2s}}.
\end{align*}
\item
If $B\leq (n\log (n)/m^{2+2s})^{1/(1+2s)} $:
\begin{align*}
\inf_{\hat{f}\in\mathcal{F}_{dist}(B,...,B)}\sup_{f_0\in B_{2,\infty}^s(L)} E_{f_0}\|\hat{f}-f_0\|_2^2 \gtrsim   \Big(\frac{n}{m\log n}\Big)^{-\frac{2s}{1+2s}}.
\end{align*}
\end{itemize}
\end{theorem}

\begin{proof}
See Section \ref{sec: minimaxL2LB}
\end{proof}

The result shows that it is indeed only possible to obtain the optimal rate $n^{-s/(1+2s)}$ 
over Besov balls of regularity $s$ if, up to a logarithmic factor, every machine is allowed to 
transmit order $n^{1/(1+2s)}$ bits to the central machine. 
The following theorem shows that this result is indeed sharp (up to log-factors), i.e.\ if  
order $n^{1/(1+2s)}$ bits are allowed, then the optimal rate can indeed be achieved with some procedure.
In fact, the theorem considers the first two cases of the preceding one, i.e.\ 
$(n\log (n)/m^{2+2s})^{1/(1+2s)}\leq B$. The third case is not interesting 
since in that case distributed methods do not perform better  than any standard technique applied 
on a single, local server.

\begin{theorem}\label{theorem: minimaxL2UB}
Let $s, L > 0$, $m\leq n$. Then there exists a distributed estimator $\hat f\in\mathcal{F}_{dist}(B,\ldots,B; B_{2,\infty}^{s}(L))$ satisfying:
\begin{itemize}
\item for $B\geq n^{1/(1+2s)}/\log n$:
\begin{align*}
\sup_{f_0\in B_{2,\infty}^{s}(L)}\mathbb{E}_{f_0}\|\hat{f}-f_0\|_2^2\lesssim n^{-\frac{2s}{1+2s}}\vee (B/\log n)^{-2s},
\end{align*}
\item for $(n\log (n)/m^{2+2s})^{1/(1+2s)}\vee \log n\leq B \leq  n^{1/(1+2s)}/\log n$:
\begin{align*}
\sup_{f_0\in B_{2,\infty}^{s}(L)}\mathbb{E}_{f_0}\| \hat{f}-f_0\|_2^2\lesssim  M_n \Big(\frac{n^{1/(1+2s)}}{B\log n}\Big)^{\frac{2s}{2+2s}}n^{-\frac{2s}{1+2s}},
\end{align*}
with $M_n=(\log n)^{2s} $.
\end{itemize}
\end{theorem}

\begin{proof}
See Section \ref{sec: minimaxL2UB}
\end{proof}

One can also derive similar matching lower and upper bounds for the $L_{\infty}$-norm for $f_0\in B_{\infty,\infty}^s(L)$ in case of the Gaussian white noise model, as in \cite{szabo:zanten:2018} where the nonparametric regression model was considered. Since our focus in this paper is not on deriving minimax rates, we have deferred this result to Section \ref{sec: minimaxLinftyLB} in the appendix.

\subsection{Simultaneous adaptation to smoothness and minimal communication}\label{sec: adapt:GWN}

In view of the preceding two theorems we can conclude that 
when the goal is to estimate $s$-smooth functions at the rate $n^{-s/(1+2s)}$, the 
 optimal, minimal number of transmitted bits is $n^{1/(1+2s)}$ (up to a logarithmic factor). Transmitting less bits will result in (polynomially) sub-optimal convergence rate for any distributed method, while by transmitting at least the optimal amount of bits one can construct distributed estimators reaching  the convergence rate of non-distributed techniques.

\subsubsection{Adaptation in $L_2$}
The procedure $\hat f$ exhibited in (the proof of) Theorem \ref{theorem: minimaxL2UB} has the desirable property that 
if $f_0 \in B_{2,\infty}^{s}(L)$, then, up to log-factors, using the minimal communication it achieves the optimal rate 
$n^{-s/(1+2s)}$. This procedure is, however, not adaptive: it uses the knowledge of the regularity level $s$ 
of the unknown function.  
In this section we investigate the more relevant question  under which conditions can we  simultaneously achieve the optimal 
convergence rate and minimal communication without using any information about the smoothness of the truth.

If we are willing to assume that the true regularity is at least $s\geq s_0$ for some known $s_0>0$
and are in addition willing to allow order $n^{{1}/({1+2s_0})}$  bits to be communicated between 
the local and the central machines, then it is straightforward to achieve adaptation to smoothness.

\begin{proposition}\label{prop: alg:adapt:s0}
Let $s_{\max}>s_0>0$, $L > 0$, $m\leq n$, and $B_0=n^{1/(1+2s_0)}\log n$. Then there exists a distributed estimator $\hat f\in\mathcal{F}_{dist}(B_0,\ldots,B_0; B_{2,\infty}^{s}(L))$ for all $s\in [s_0,s_{\max}]$ satisfying that
\begin{align*}
\sup_{s\in [s_0,s_{\max}]}\sup_{f_0\in B_{2,\infty}^{s}(L)}n^{\frac{2s}{1+2s}}\mathbb{E}_{f_0}\|\hat{f}-f_0\|_2^2=O(1).
\end{align*}
\end{proposition}

\begin{proof}
See Section \ref{sec: alg:adapt:s0}
\end{proof}

The problem with the above method is that it always transmits a multiple of $n^{{1}/{1+2s_0}}\log n$ bits, which can be  substantially more than the optimal $n^{1/(1+2s)}$ if the true smoothness $s$ happens to be larger 
than the assumed lower bound $s_0$. The question naturally arises: is it possible to achieve adaptation to 
smoothness while at the same time automatically transmitting the minimal amount of bits?

 We show that in case of the $L_2$-norm one can only adapt up to a limited range of smoothness levels (depending on the number of local machines), and beyond that one will achieve sub-optimal rates (where the rate is sub-optimal by a polynomial factor).

\begin{theorem}\label{thm: adaptL2negative}
Suppose that $m=n^p$ for some $p\in (0,1/2)$. Then for any regularity parameters $s_2>s_1>1/(4p)-1/2$ there does not exist 
a distributed method which adapts to the number of transmitted bits and at the same time achieves the minimax risk as well, i.e. it is not possible to simultaneously obtain a distributed method with $\hat{B}^{(i)}\leq n^{1/(1+2s_1)+\eps_1}\log n$ and for $l=1,2$ that
\begin{align}
&\sup_{i\in\{1,...,m\}}\sup_{f_0\in  B_{2,\infty}^{s_l}(L)}  E_{f_0}^{(i)} \hat{B}^{(i)}\lesssim n^{\frac{1}{1+2s_l}+\eps_1}\quad\text{and}\label{eq:adapt:1}\\
&\sup_{f_0\in  B_{2,\infty}^{s_l}(L)} E_{f_0} \| \hat{f}-f_0\|_2^2\lesssim n^{-\frac{2s_l}{1+2s_l}+\eps_2},\label{eq:adapt:2}
\end{align}
for some small enough constants $\eps_1,\eps_2>0$ depending only on $s_1,s_2$ and $p$. 
\end{theorem}

\begin{proof}
See Section \ref{sec: adaptL2negative}.
\end{proof}

% I feel this one only adds a bit confusion:

%\begin{remark}
%The technical condition $B^{(i)}\leq n^{1/(1+2s_1)+\eps_1}\log n$ in Theorem \ref{thm: adaptL2negative} is mild since from assumption \eqref{eq:adapt:1} it follows with probability tending to one. We also show in the next theorem considering adaptation for less regular function that such an assumption does not introduce any additional restrictions and follows naturally.
%\end{remark}

The above theorem tells us that considering even just two regularity classes (with regularities above some threshold level) there doesn't exist any distributed method, which transmits  the optimal amount of bits multiplied by some (small) polynomial factor and reaches the minimax rate in both smoothness classes up to a (small) polynomial factor. The above negative results deliver a strong message, as the question of non-existence can not be resolved by allowing extra logarithmic factors, but is on the polynomial level.

The phenomenon behind the negative result is that in case of many local machines (large  $m$) it is getting more difficult to test locally between the regularity classes (as the local ``sample size'' decreases in $m$) and also the ``local regularity'' of the function which one can judge at noise level $m/n$ might be completely different than the ``global regularity'' of the truth which can be judged at a smaller noise level $1/n$.

Although full adaptation is not possible, it turns out that on a limited range of regularity levels it is possible to construct adaptive methods. Below we derive the complement of the preceding result and show that for regularities
below the threshold  $1/(4p)-1/2$ we {\em can} adapt 
 to smoothness and transmit the minimal  number of  bits at the same time.

\begin{theorem}\label{thm: adaptL2positive}
For arbitrary $0< s_1<s_2\leq 1/(4p)-1/2$ and $m
\geq 5\log n$ there exists a distributed estimator $\hat{f}$ with number of transmitted bits $(\hat{B}^{(1)},...,\hat{B}^{(m)})$, such that $\hat{B}^{(i)}\leq n^{\frac{1}{1+2s_1}}\log n$, $i=1,...,m$, and for all $s\in\{s_1,s_2\}$
\begin{align*}
&\max_{i\in\{1,...,m\}}\sup_{f_0\in B_{2,\infty}^{s}(L)}E_{f_0}^{(i)}\hat{B}^{(i)}\lesssim C_2 n^{\frac{1}{1+2s}}\log n\quad\text{and}\\
&\sup_{f_0\in B_{2,\infty}^{s}(L)}E_{f_0}\|\hat{f}-f_0\|_2^2\lesssim  n^{-\frac{2s}{1+2s}}.\\
\end{align*}
\end{theorem}

\begin{proof}
See Section \ref{sec: adaptL2positive}.
\end{proof}

The proposed procedure has two stages. First we ``estimate'' the smoothness of the underlying functional parameter of interest in every local machine parallel to each other  and based on that transmit the right amount of information to the central machine. In the second stage we aggregate the locally transmitted information and provide a ``global'' adaptive estimator. The difficulty, as also discussed above, arises from the higher noise level in the local problems which results in less accurate tests between the smoothness classes. The existence of an estimator which can achieve adaptation (in a limited range of smoothness classes) is a consequence of 
the difference between the nonparametric testing and estimation rates in the case of the $L_2$-norm, see for instance \cite{ingster2,gine:nickl:2016}.  Since one can test between smoothness classes with a faster rate than the corresponding estimation rate, it can compensate (up to some extent) for the higher local noise level $m/n$.

The preceding result can be extended to a scale of smoothness classes as well.

\begin{corollary}\label{thm: adaptL2positive2}
Assume that $m=n^{p}$ for some $0<p\leq 1/2$, then for arbitrary $0<s_1<s_2<1/(4p)-1/2$ and $m\geq 5\log n$ there exists a distributed estimator $\hat{f}$ transmitting $\hat{B}^{(i)}$ bits in the local machines $i=1,...,m$ satisfying  that $\hat{B}^{(i)}\leq n^{\frac{1}{1+2s_1}}\log n$ and
\begin{align*}
&\max_{i=1,...,m}\sup_{s\in [s_1, s_2]}\sup_{f_0\in B_{2,\infty}^{s}(L)} \frac{E_{f_0}^{(i)}\hat{B}^{(i)}}{n^{1/(1+2s)}\log n }\lesssim 1\quad\text{and}\\
&\sup_{s\in [s_1, s_2]}\sup_{f_0\in B_{2,\infty}^{s}(L)}  \frac{E_{f_0}\|\hat{f}-f_0\|_2^2}{n^{-2s/(1+2s)}}\lesssim 1.
\end{align*}
\end{corollary}

\begin{proof}
See Section \ref{sec: adaptL2positive2}.
\end{proof}

 The idea of the proof of this corollary is to introduce a grid of regularities in the interval $[s_1,s_2]$ and test 
 between which two grid points the true regularity lies. Then one can apply the distributed method introduced in the proof of Theorem \ref{thm: adaptL2positive} to derive the stated results.

\subsubsection{Adaptation in $L_\infty$}

Next we deal with the $L_{\infty}$-norm case. Here we show that in contrast to the $L_{2}$-case, adaptation is not possible even on a limited range of smoothness classes. The reason behind it is that in this case the minimax testing and estimation rates are the same and hence there is no room left to compensate for the higher local noise level.

\begin{theorem}\label{thm: adaptLinftynegative}
Take any $0<s_1<s_2$ and
assume that $m=n^{p}$, with $p\in(0,1/2)$. Then  there does not exist a distributed estimator $\hat{f}$ with transmitted bits $\hat{B}^{(i)}\leq n^{\frac{1}{1+2s_1}+\eps_1}$, $i=1,...,m$, satisfying that for $\ell=1,2$
\begin{align}
&\max_{i=1,...,m}\sup_{f_0\in B_{\infty,\infty}^{s_\ell}(L)}E_{f_0}^{(i)}\hat{B}^{(i)}\lesssim  n^{\frac{1}{1+2s_\ell}+\eps_1},\quad\text{and}\label{eq:UB:hatB}\\
&\sup_{f_0\in B_{\infty,\infty}^{s_\ell}(L)}E_{f_0} \|\hat{f}-f_0\|_\infty \lesssim (n/\log n)^{-\frac{s_\ell}{1+2s_\ell}+\eps_2},\label{eq:UB:risk}
\end{align}
for some sufficiently small $\eps_1,\eps_2>0$.

\end{theorem}

\begin{proof}
See Section \ref{sec: adaptLinftynegative}

\end{proof}

Next we introduce some additional restriction on the true function of interest under which adaptation is possible in the distributed setting. To do so we consider the so-called self-similarity assumption, where loosely speaking  we assume that the true function has similar smoothness at every resolution level. This will allow us to estimate  the regularity $s$ of the functional parameter of interest and therefore transmit the right amount of bits  from the local machines to the central one.

We first introduce necessary notation. Let $\psi_{jk}$ be the wavelet basis functions described in Appendix \ref{sec: wavelets}.
For $f \in L_2[0,1]$ and natural numbers $j_1 \le j_2$ we define
\begin{align*}
 f_{[j_1,j_2]}=\sum_{j=j_1}^{j_2}\sum_{k=1}^{2^j}f_{jk}\psi_{jk}.
\end{align*}
Then following \cite{bull:2012} we say that the function $f\in B_{\infty,\infty}^s(L)$ belongs to the self-similar class $ S^s_{\infty}(L,\eps,j_0,\rho)$ if,
\begin{align}
\|f_{[j,\rho j]}\|_{B_{\infty,\infty}^s}\geq \eps L,\quad\text{for $j\geq j_0$ and $\rho>1$.}\label{cond: selfsim:Linf}
\end{align}

The self-similarity property was introduced (amongst other places) in the context of adaptive confidence bands. It was shown that under self-similarity one can construct adaptive $L_{\infty}$ confidence bands whose size also adapts to the level of regularity, see for instance \cite{picard:2000,gine:nickl:2010,bull:2012}. The underlying idea is the same as here. Under this assumption one can provide a consistent estimator for the smoothness and based on that construct the band corresponding the function class.

The following theorem shows that under the self-similarity assumption there exists a distributed method which adapts to 
regularity and at the same time  transmits the minimal amount of bits (again up to logarithmic factors).

\begin{theorem}\label{thm: adaptLinftypositive}
Consider the distributed Gaussian white noise model with $m\leq n^{\delta}$, for some $\delta\in(0,1)$ and assume that $f_0\in B_{\infty,\infty}^{s}(L)$ for some $s\in[s_1,s_2]$ (where $0<s_1<s_2$ are arbitrary). 
Then there exists a distributed method such that the number of transmitted bits satisfies $\hat{B}^{(i)}\leq (n/\log n)^{1/(1+2s_1)}\log n$
and
\begin{align*}
&\max_{i\in\{1,...,m\}}\sup_{s\in[s_1,s_2]}\sup_{f_0\in S^s_{\infty}(L,\eps,j_0,\rho)}\frac{E_{f_0}^{(i)}\hat{B}^{(i)}}{ n^{\frac{1}{1+2s}}(\log n)^{\frac{2s}{1+2s}}}\lesssim 1\quad\text{and}\\
&\sup_{s\in[s_1,s_2]}\sup_{f_0\in S^s_{\infty}(L,\eps,j_0,\rho)}\frac{E_{f_0} \|\hat{f}-f_0\|_{\infty}}{(n/\log n)^{-\frac{s}{1+2s}}}\lesssim 1.
\end{align*}
\end{theorem}

\begin{proof}
See Section \ref{sec: adaptLinftypositive}
\end{proof}

%\section{Simulation study}
%Shall we carry out a simulation study? For instance we could compare the performance of the distributed and non-distributed methods for estimating the functional parameter of interest.
%

%
%\section{Discussion}
%
%The above observed phenomena is closely related to the existence of adaptive and honest confidence sets. In the $L_2$-norm case one can construct confidence sets which have optimal minimax size and at the same time contain the true underlying function of interest with the prescribed high probability, uniformly in case the true regularity belongs to a limited range of smoothness classes $s\in [s_0,2s_0]$ for some given $s_0>0$. Outside of this interval it is not possible to construct confidence sets which adapt to the minimax rate and contain the truth with uniformly high probability, see for instance \cite{robins:2006,bull:2013}.  The existence of the limited range of smoothness classes is also due to the different $L_2$-testing and estimation rates. At the same time it is impossible to construct adaptive and honest confidence sets for the $L_{\infty}$-norm even over a limited range of smoothness classes. This is due also to the matching testing and estimation rates, see for instance \cite{cai:low:04,gine:nickl:2010,bull:2012}.
%
%

\section{Proofs for the adaptation results}

In the proofs we will work with the wavelet decomposition of the functional parameter $f_0$. In our analysis we consider the Daubechie wavelets $\psi_{jk}(t)$ for $j=0,1,...$, $k=1,...,2^j$, $t\in[0,1]$ and denote by $f_{0,jk}=\int_0^1\psi_{jk}(t)f_0(t) dt$ the corresponding wavelet coefficients.  In Section \ref{sec: wavelets} we have collected a few properties of Daubechie wavelets which we will apply throughout the proofs.

We note that following from the orthonormality of the Daubechie wavelets we have 
that the Gaussian white noise model can be written in the sequence representation
\begin{align}
X_{jk}^{(i)}=f_{0,jk}+\sqrt{\frac{m}{n}} Z_{jk}^{(i)},\quad j=0,1,2,...;\, k=1,...,2^j;\, i=1,...,m,\label{model: sequence}
\end{align}
where $X_{jk}^{(i)}$, $j=0,1,...,$ $k=1,...,2^j$ are the noisy observations $X_{jk}^{(i)}=\int_0^1 \psi_{jk}(t)dX^{(i)}(t)$ and $Z_{jk}^{(i)}$ are iid standard normal random variables.

\subsection{Proof of Proposition \ref{prop: alg:adapt:s0}}\label{sec: alg:adapt:s0}
Consider  the sequence representation of the distributed Gaussian white noise model, see \eqref{model: sequence} using  at least $s_{\max}$ regular Daubechie wavelets. Then by transmitting $n^{1/(1+2s_0)}\log n$ bits (the first $n^{1/(1+2s_0)}$ elements of the sequence representation of the model up to the first $0.5\log n$ digits in the binary representation of the number, see Algorithm \ref{alg: transmit:number}) to the central machine and averaging the transmitted local data we arrive to the global sequence model
\begin{align*}
Y_{jk}=f_{0,jk}+\sqrt{\frac{1}{n}}Z_{jk}+\eps_{jk},\quad j=0,...,\lfloor\frac{\log n}{1+2s_0}\rfloor,\, k=1,...,2^j,
\end{align*}
where $Z_{jk}$ are iid standard Gaussian random variables and $|\eps_{jk}|\leq n^{-1/2}$ are random variables representing the error term arising from transmitting only the first $0.5\log n$ digits of the observations. These error terms are in fact negligible. Then using arbitrary adaptation technique (for instance Lepski's method \cite{lepski:1991}) one can construct an estimator $\hat{f}$ achieving the minimax risk for every $f\in B_{2,\infty}^s(L), s_0\leq s\leq s_{\max}$. 

\subsection{Proof of Theorem \ref{thm: adaptL2negative}}\label{sec: adaptL2negative}
We argue by contradiction. We assume that the inequalities \eqref{eq:adapt:1} and \eqref{eq:adapt:2} hold. Then we construct a finite but large enough set $\mathcal{F}_0\subset  B_{2,\infty}^{s_1}(L)$ such that there does not exist a  consistent test between the elements of the set and the zero function, which clearly belongs to the smoother class $B_{2,\infty}^{s_2}(L)$. Using this non-existence result we arrive to contradiction with our assumptions. 

As a first step we construct the set $\mathcal{F}_0$. Let us introduce the following notations
\begin{align}
\tilde\delta_n&=\bar\delta_n\wedge (n/m)^{-\frac{1+2s_1}{1/2+2s_1}}n^{-\eps_3},\quad \text{with}\label{def: tilde:delta_n}\\
\bar{\delta}_n&=\min\Big\{\frac{m}{n \log m }, \frac{1}{n [\bar\delta_n^{1/(1+2s_1)} \beta_n \wedge  1]  \log m} \Big\},\nonumber\\
\beta_n&=(\Gamma_n \vee n^{\eps_1-\frac{(s_1+1/4)\eps_3}{1+2s_1}} n^{\frac{1}{1+2s_1}})\log n \quad\text{and}\quad \Gamma_n=  n^{\frac{1/2}{1+2s_1}+\frac{1/2}{1+2s_2}+\eps_1},
\nonumber
\end{align}
and constants $\eps_3\in(0,\frac{p(1+2s_1)-1/2}{1/2+2s_1})$, where $p(1+2s_1)-1/2>0$ follows from the assumption $s_1>1/(4p)-1/2$, and $\eps_1\in\big(0, \frac{s_2-s_1}{(1+2s_1)(1+2s_2)}\wedge\frac{(s_1+1/4)\eps_3}{1+2s_1}\big)$.
Note that $\beta_n\leq n^{\frac{1}{1+2s_1}-\eps_4}\log n$, with $\eps_4=(\frac{s_2-s_1}{(1+2s_1)(1+2s_2)}-\eps_1)\wedge \big(\frac{(s_1+1/4)\eps_3}{1+2s_1}-\eps_1\big)>0$. In view of the definition of $\bar{\delta}_n$ this implies that $\bar\delta_n\geq n^{\eps_4\frac{1+2s_1}{2+2s_1}}/(n\log n)\gg n^{-1+\eps_4/2}$.

Furthermore, $$(n/m)^{-\frac{1+2s_1}{1/2+2s_1}}n^{-\eps_3}= n^{-(1-p)\frac{1+2s_1}{1/2+2s_1}-\eps_3} =
n^{-1} n^{\frac{p(1+2s_1)-1/2}{1/2+2s_1}-\eps_3}.$$
Therefore we can conclude that for large enough $n$
\begin{align}
\tilde{\delta}_n\geq n^{-1+\eps_5}\quad\text{with}\quad\text{$\eps_5= (\eps_4/2) \wedge \Big(\frac{p(1+2s_1)-1/2}{1/2+2s_1}-\eps_3\Big)>0.$}\label{eq: help:LB_delta}
\end{align}

 The elements $f\in\mathcal{F}_0$ are then defined with the wavelet coefficients as
 \begin{align}
 f_{jk}=
 \begin{cases}
\beta_{k}\tilde{\delta}_n^{1/2}, & \text{if}\quad j=j_n:=\lfloor \frac{\log \tilde{\delta}_n^{-1}}{1+2s_1}\rfloor , k=1,...,2^{j_n},\\
 0, & \text{else}, 
 \end{cases}
\end{align}
where $\beta_{k}\in\{-1,1\}$. It is easy to check that $\mathcal{F}_0\subset B_{2,\infty}^{s_1}(1)$ and besides, for every $f\in\mathcal{F}_0$, in view of the definition of $\tilde\delta_n$,
\begin{align*}
\|0-f\|_2^2=\sum_{j=0}^{\infty}\sum_{k=1}^{2^j}f_{jk}^2= 2^{j_n}\tilde{\delta}_n\leq \tilde{\delta}_n^{\frac{2s_1}{1+2s_1}}=o \big((n/m)^{-\frac{2s_1}{1/2+2s_1}}\big).
\end{align*}

Next we take the average likelihood ratio over the class $\mathcal{F}_0$
\begin{align*}
Z=\frac{1}{|\mathcal{F}_0|}\sum_{f\in\mathcal{F}_0}\frac{dP_{f}^{(i)}}{d P_0^{(i)}},\quad \text{where $| \mathcal{F}_0|= 2^{j_n}$}.
\end{align*}
In view of  (6.23) of \cite{gine:nickl:2016}
\begin{align}
\inf_{\Psi^{(i)}} \{E_{0}^{(i)}\Psi^{(i)} +\frac{1}{|\mathcal{F}_0|}\sum_{f\in\mathcal{F}_0}E_f^{(i)} (1-\Psi^{(i)})\}\geq (1-\eta_{n})\Big(1-\frac{\sqrt{E_{0}^{(i)}(Z-1)^2}}{\eta_{n}} \Big),\label{eq: test:nonconsistent}
\end{align}
for every $\eta_{n}\in(0,1)$, where the infimum is taken over all local tests in the local problems. Furthermore one can show by following the steps in the proof of Theorem 6.2.11 c) on pages 493-494 of \cite{gine:nickl:2016} (with $\gamma_{n/m}'=c_0^2 (n/m)\tilde{\delta}_n$ and $\gamma_{n/m}=  (n/m)\tilde{\delta}_{n}^{\frac{1/2+2s_1}{1+2s_1}}\leq n^{-\frac{1/2+2s_1}{1+2s_1}\eps_3}$) that 
\begin{align*}
E_{0}^{(i)}(Z-1)^2\leq  \exp\{c'\gamma_{n/m}^2\}-1\lesssim\gamma_{n/m}^2\lesssim n^{-\frac{1+4s_1}{1+2s_1}\eps_3}.
\end{align*}
By choosing  $\eta_{n}=  n^{-\frac{(1/4+s_1)\eps_3}{1+2s_1}}$  we get that
\begin{align}
\inf_{\Psi^{(i)}} \{E_{0}^{(i)}\Psi^{(i)} +\frac{1}{|\mathcal{F}_0|}\sum_{f\in\mathcal{F}_0}E_f^{(i)} (1-\Psi^{(i)})\}\geq (1-C\eta_{n})^2,\label{eq: test:nonconsistent2}
\end{align}
 for some large enough constant $C>0$, concluding the proof of the non-existence of consistent tests between $\mathcal{F}_0$ and the zero function.

Next we show that \eqref{eq: test:nonconsistent2} contradicts our assumptions. Let us define the test
\begin{align*}
\Psi^{(i)}=1_{\hat{B}^{(i)}\geq \Gamma_n}.
\end{align*}
First note that following from Markov's inequality and assumption \eqref{eq:adapt:1}
\begin{align*}
E_0^{(i)}\Psi^{(i)}=P_0^{(i)}(\hat{B}^{(i)}\geq \Gamma_n)\leq E_0^{(i)}(\hat{B}^{(i)})/\Gamma_n\leq n^{\frac{1/2}{1+2s_2}-\frac{1/2}{1+2s_1}}=o(1).
\end{align*}
Therefore in view of \eqref{eq: test:nonconsistent2} we have that
\begin{align*}
\frac{1}{|\mathcal{F}_0|}\sum_{f\in\mathcal{F}_0} P_f^{(i)}(\hat{B}^{(i)}<\Gamma_n) &=\frac{1}{|\mathcal{F}_0|}\sum_{f\in\mathcal{F}_0} E_f^{(i)}(1-\Psi^{(i)})\\
&\geq (1- C\eta_n)^2- n^{\frac{1/2}{1+2s_2}-\frac{1/2}{1+2s_1}}.
\end{align*} 
As a consequence and in view of assumption $\hat{B}^{(i)}\leq  C n^{\frac{1}{1+2s_1}+\eps_1}\log n$
\begin{align*}
\frac{1}{|\mathcal{F}_0|}\sum_{f\in\mathcal{F}_0} E_f^{(i)} \hat{B}^{(i)}\lesssim  \Gamma_n+ n^{\frac{1}{1+2s_1}+\eps_1}(\log n)( \eta_n+ n^{\frac{1/2}{1+2s_2}-\frac{1/2}{1+2s_1}})\lesssim \beta_n.
\end{align*}
This means that the expected number (with respect to the joint distribution of the variables $F$ and $P_f$, $f\in\mathcal{F}_0$) of transmitted bits on the class $\mathcal{F}_0$ is bounded from above by a multiple of $\beta_n$. So the distributed estimator satisfies assertion \eqref{eq: UB:info:length} in the proof of Theorem \ref{thm: minimaxL2LBgeneral} with $B^{(i)}$ replaced by $C\beta_n$. Hence in view of the minimax lower bound derived in assertion \eqref{eq: LB:average:bits} and the definition of $\tilde{\delta}_n$ (with $B^{(i)}$ replaced by $\beta_n$ in the definition of $\delta_n$ in the proof of Theorem \ref{thm: minimaxL2LBgeneral})
\begin{align*}
\sup_{f_0\in \mathcal{F}_0} E_{f_0}\|\hat{f}-f_0\|_2^2 \gtrsim \tilde{\delta}_n^{\frac{2s_1}{1+2s_1}}\gg n^{-\frac{2s_1}{1+2s_1}+\eps_2},
\end{align*}
with $\eps_2=2\eps_5 s_1/(1+2s_1)$, where the last inequality follows from \eqref{eq: help:LB_delta}. 
This contradicts assumption \eqref{eq:adapt:2}, finishing the proof of our statement.

\subsection{Proof of Theorem \ref{thm: adaptL2positive}}\label{sec: adaptL2positive}
In our proof we work with the equivalent sequence representation of the model \eqref{model: sequence}. As a first step we split the data in all of the local models $i\in\{1,...,m\}$ into two subsets $X_{jk}^{(i,1)},X_{jk}^{(i,2)}$ for $j=0,1,2,..,$ $k=1,...,2^j$,  such that they are pairwise independent and their variance is $2m/n$ (this can be done by adding and substracting $\tilde{Z}_{jk}^{(i)}\stackrel{iid}{\sim} N(0,m/n)$ from $X_{jk}^{(i)}$).  Let us then denote by $P_{X^{(i,1)}}$ and $P_{X^{(i,2)}}$  the distribution of the first and second subset of observations, respectively, and by $P_{X^{(i,2)}|X^{(i,1)} }$ the conditional distribution of the second subset given the first. The corresponding expected values are denoted by $E_{X^{(i,1)}}, E_{X^{(i,2)}}$, and $E_{X^{(i,2)}|X^{(i,1)}}$, respectively. Finally let us introduce the notations  $X_l=(X^{(1,l)},...,X^{(m,l)})$, $l=1,2$ and denote by $P_{X_l}$ and $E_{X_l}$ the corresponding probability distributions and expected values.

Next note that it was shown in \cite{carpentier:2015} that there exists a consistent composite test between the classes $B_{2,\infty}^{s_2}(L)$ and $B_{2,\infty}^{s_1}(L)$  in the local problem using the first subset of observations $X^{(i,1)}$ if they are at least $(n/m)^{-s_1/(1/2+2s_1)}$ separated. The test proposed in Section 3 of  \cite{carpentier:2015} takes the form (in the local machines using the first subset of observations $X^{(i,1)}$)
\begin{align}
\Psi_{n/m}^{(i)}=\Psi_{n/m}^{(i)}(\alpha,s_1,s_2)&=1-\prod_{0\leq l\leq \lfloor\frac{\log (n/(2m))}{2s_1+1/2}\rfloor}1_{\{ T_{n/m}^{(i)}(l)\leq t_{n/m}(l,s_2,\alpha)\}},\label{def: test}
\end{align}
where
\begin{align*}
t_{n/m}(l,s_2,\alpha)&= \frac{L^2}{2^{2ls_2}}+\frac{L}{2^{ls_2}}\tau_l+\frac{\tau_l^2}{4}, \\
 \tau_l &=24\sqrt{\frac{z_0}{\alpha}}\frac{2^{l+\big\lfloor\frac{\log (n/(2m))}{1/2+2s_2}\big\rfloor}}{\sqrt{n/(2m)}}, \quad \text {for $l>0$},\\
 \tau_0 &=24\sqrt{\frac{z_0}{\alpha}}\frac{1}{\sqrt{n/(2m)}},\nonumber\\
T_{n/m}^{(i)}(l)&= \| \Pi_l\hat{f}_{n/m}^{(i)}\|_2^2-m2^{l+1}/n, \quad \text {for $l>0$}, \nonumber\\
T_{n/m}^{(i)}(0)&= \| \Pi_0\hat{f}_{n/m}^{(i)}\|_2^2-2mz_0/n,
\end{align*}
where $\Pi_l f$ denotes the projection of the function $f$ to the resolution level $l$, i.e. $\Pi_l f=\sum_{k=1}^{2^j}f_{lk}\psi_{l,k}$, see (3.1) and (3.2) of  \cite{carpentier:2015}, $\hat{f}_{n/m}^{(i)}$ is the wavelet estimate of $f$ in the $i$th local machine using observations $X^{(i,1)}$, see the top of page 6 of \cite{carpentier:2015}, and $z_0=1$ (since for notational convenience we take $J_0=0$, see Section \ref{sec: wavelets}, we have $z_0=2^{J_0}=1$). Let us introduce the notation $R_{\alpha}^{s_1}(L)= \{f\in B_{2,\infty}^{s_1}(L):\, \|f-B_{2,\infty}^{s_2}(L)\|_2\geq \tilde{C}_\alpha (n/m)^{-\frac{s_1}{1/2+2s_1}}\}$. %Note that for $s_1'<s_1<s_2$ we have that $\psi_{n/m}^{(i)}(\alpha,s_1,s_2)\leq \psi_{n/m}^{(i)}(\alpha,s_1',s_2)$.

In view of Lemma \ref{lem: redo:carpentier} we have for all $\alpha\in(0,1)$ and $0<m\leq n$ that
\begin{align}
\sup_{f\in B_{2,\infty}^{s_2}(L)}E_{X^{(i,1)}}\Psi_{n/m}^{(i)} + \sup_{f\in R_{\alpha}^{s_1}(L)}E_{X^{(i,1)}} (1-\Psi_{n/m}^{(i)})\leq c e^{-0.5/\sqrt{\alpha}},\label{eq:cons:test}
\end{align}
with $\tilde{C}_\alpha=24(\frac{2^{s_1}L}{\sqrt{1-2^{-2s_1}}}+19)2^{\frac{s_1}{1+2s_1}}/\sqrt{\alpha}$ and $c$ not depending on $\alpha,n,m$. Let $M_n= n^{\frac{2s_1(1/2-p(1+2s_1))}{(1+2s_1)(1/2+2s_1)}}$ tending to infinity (where the positivity of the exponent follows form the assumption $s_1< 1/(4p)-1/2$). Then there exists a consistent test $\Psi^{(i)}_{n/m}$ (with $\alpha=M_n^{-1}$) in each local problem between the hypotheses
$$H_0:\, f\in B_{2,\infty}^{s_2}(L)\quad\text{vs}\quad H_1:\, f\in R_{M_n^{-1}}^{s_1}(L).$$
Using the test function above, we define the smoothness estimate as
 $$\hat s_{n/m}^{(i)}=\begin{cases} s_2,& \text{if $\Psi_{n/m}^{(i)}=0$},\\
s_1,& \text{if $\Psi_{n/m}^{(i)}=1$}.
\end{cases}
$$ 
In each local model we take the first $n^{1/(1+2\hat{s}_{n/m}^{(i)})}$ coefficients in the second subset of observations in the sequence representation, i.e. $X_{jk}^{(i,2)}$ with $2^j+k\leq n^{1/(1+2\hat{s}_{n/m}^{(i)})}$. Since these numbers might note have a finite binary representation we transmit their approximations $Y_{jk}^{(i)}$ following Algorithm \ref{alg: transmit:number}. Note that in view of  Lemma \ref{lem: approx} (with $\mu=f_{0,jk}$) we have that $l(Y_{jk}^{(i)})\leq \log n$ with approximation error $|\eps_{jk}^{(i)}|= |X_{jk}^{(i,2)}-Y_{jk}^{(i)}|\leq n^{-1/2}$ on a set $\mathcal{E}_{jk}^{(i)}$ with $P_{X^{(i,2}}\big((\mathcal{E}_{jk}^{(i)})^c\big)\leq e^{-c'n}$, for some $c'>0$. Let us then introduce the notation 
\begin{align}
\mathcal{E}=\cap_{i=1}^m\cap_{j=0}^{\log n}\cap_{k=1}^{2^j}\mathcal{E}_{jk}^{(i)}\label{def: setE}
\end{align}
 and note that $P_{X_2}(\mathcal{E}^c)\leq n^2  e^{-c'n}\lesssim e^{-cn}$, for any $0<c<c'$.  Hence the  number of transmitted bits conditioned on the first subsample $X^{(i,1)}$ is bounded from above by $l(Y^{(i)})\leq n^{1/(1+2\hat{s}_{n/m}^{(i)})}\log n$ almost surely.

Let us denote by $\tilde{N}$ the median of the values $n^{1/(1+2\hat{s}_{n/m}^{(i)})}$, $i=1,...,m$ and $\hat{s}$ the corresponding regularity estimator. Then we construct our estimator $\hat{f}$ as the average of the transmitted observations (for the first $\tilde{N}$ coefficient), i.e.
\begin{align*}
\hat{f}_{n,jk}=
\begin{cases}
 \frac{1}{|M_{jk}|}\sum_{i\in M_{jk}} Y_{jk}^{(i)},& 2^{j}+k\leq \tilde{N},\\
 0, & \text{for}\,\,\, 2^{j}+k>\tilde{N},
\end{cases} 
\end{align*}
where $M_{jk}$ is the collection of local machines satisfying $2^j+k\leq n^{1/(1+2\hat{s}_{n/m}^{(i)})}$, i.e. the machines from which the local approximations $Y_{jk}^{(i)}$ are  transmitted. 

We show that this procedure achieves the minimax convergence rate and transmits the optimal amount of bits (up to a logarithmic factor). First note that $\hat{B}^{(i)}\lesssim n^{1/(1+2s_1)}\log n$ follows immediately by construction.  Then recall that the test $\Psi_{n/m}^{(i)}$ is consistent, hence 
$$\sup_{f\in B_{2,\infty}^{s_2}(L)} P_{X^{(i,1)}}(\hat{s}_{n/m}^{(i)}=s_1) \leq C e^{-M_n^{1/2}/2}$$
 and
\begin{align*}
\sup_{f\in B_{2,\infty}^{s_2}(L)}E_{X^{(i,1)},X^{(i,2)}}\hat{B}^{(i)}
&\leq \sup_{f\in B_{2,\infty}^{s_2}(L)} E_{X^{(i,1)}}n^{1/(1+2\hat{s}_{n/m}^{(i)})}\log n\nonumber\\
&\leq n^{1/(1+2s_2)}\log n+ Ce^{-M_n^{1/2}/2}n^{1/(1+2s_1)}\log n\nonumber\\
 &\leq (1+o(1))n^{1/(1+2s_2)}\log n,
\end{align*}
verifying that the number of transmitted bits is indeed optimal.

Next we provide optimal upper bounds for the risk. First let us consider the case $f\in B_{2,\infty}^{s_2}(L) \cup R_{M_n^{-1}}^{s_1}(L)$, where the estimator $\hat{s}_{n/m}^{(i)}$ is consistent, i.e. $\hat{s}_{n/m}^{(i)}=s_1$ for $ f\in R_{M_n^{-1}}^{s_1}(L)$  and $\hat{s}_{n/m}^{(i)}=s_2$ for $f\in B_{2,\infty}^{s_2}(L)$, with $P_{X^{(i,1)}}$-probability at least $1-ce^{-M_n^{1/2}/2}$.
Let us introduce  the notation $M$ for the number of machines in $\{1,...,m\}$, where the $\hat{s}_{n/m}^{(i)}\neq s_l$, $l=2,1$, for  $f\in B_{2,\infty}^{s_2}(L) $ or $f\in R_{M_n^{-1}}^{s_1}(L)$, respectively.  Note that $M$ has a binomial distribution with parameters $m$ and $p\leq ce^{-M_n^{1/2}/2}$. Then by Hoeffding's inequality 
\begin{align}
\sup_{f\in R_{M_n^{-1}}^{s_1}(L)} P_{X_1}(\tilde{N}\neq n^{\frac{1}{1+2s_1}})+ \sup_{f\in B_{2,\infty}^{s_2}(L)}& P_{X_1}(\tilde{N}\neq n^{\frac{1}{1+2s_2}})\nonumber\\
 &\leq P\big( M\geq m/2 \big)< e^{-m/5}.\label{eq: help:UB:inconsistency}
\end{align}
Then in view of the almost sure inequality $\tilde{N}\leq n^{1/(1+2s_1)}$ we have that
\begin{align}\label{eq: UB:tildeN}
\sup_{f\in R_{M_n^{-1}}^{s_1}(L)} E_{X_1}\tilde{N}^{-2s_1}&=n^{-\frac{2s_1}{1+2s_2}}P_{X_1}(M\geq m/2)+n^{-\frac{2s_1}{1+2s_1}}P_{X_1}(M< m/2)\\
&\leq (1+o(1))n^{-\frac{2s_1}{1+2s_1}},\nonumber\\
 \sup_{f\in B_{2,\infty}^{s_2}(L)} E_{X_1}\tilde{N}&= n^{1/(1+2s_2)}P_{X_1}(M<m/2)+n^{1/(1+2s_1)}P_{X_1}(M\geq m/2)\nonumber\\
&\leq n^{1/(1+2s_2)}+ n^{1/(1+2s_1)} e^{-m/5}\leq (1+o(1))n^{1/(1+2s_2)},\nonumber
\end{align}
for $m\geq  5\log n\geq   \frac{ 10(s_2-s_1)}{(2s_1+1)(2s_2+1)}\log n$.

Then similarly to the proof of Theorem \ref{theorem: minimaxL2UB} (with $m$ replaced by $|M_{jk}|$) we get on the set $\mathcal{E}$ (with $P_{X_2}(\mathcal{E}^c)\leq e^{-cn}$), that 
\begin{align*}
\hat{f}_{n,jk}=f_{0,jk}+\frac{1}{\sqrt{n}}Z_{jk}+\eps_{jk},
\end{align*}
with $Z_{jk}\stackrel{iid}{\sim}N(0,\sqrt{2m/|M_{jk}|})$ and $|\eps_{jk}|\leq n^{-1/2}$. Also note that $|\hat{f}_{n,k}|\leq \sqrt{n}$, since $|Y^{(i)}_{jk}|\leq \sqrt{n}$ for all $i,j,k$. Using this reformulation of the estimator and the notation $\tilde{j}_n=\lfloor\log \tilde{N}\rfloor$ we get that
\begin{align}\label{eq:UB_risk_ring1}
\sup_{f\in B_{2,\infty}^{s_l}(L)} E_{X_2|X_1}\|\hat{f}-f_0\|_2^21_{\mathcal{E}}&\leq  \sum_{j\geq \tilde{j}_n}\sum_{k=1}^{2^j}f_{0,jk}^2+\sum_{j=0}^{\tilde{j}_n}\sum_{k=1}^{2^j}E(\frac{1}{\sqrt{n}}Z_{jk}+\eps_{jk})^21_\mathcal{E}\\
&\leq  \sum_{j\geq \tilde{j}_n}2^{-2js_l}\sup_{j\geq \tilde{j}_n} 2^{2js_l}\sum_{k=1}^{2^j}f_{0,jk}^2+
\sum_{j=0}^{\tilde{j}_n}\sum_{k=1}^{2^j}\frac{2E (Z_{jk}^2) }{n}+ \frac{2}{n} \nonumber \\
&\lesssim 2^{-2j_ns_l} +2^{\tilde{j}_n}/n\asymp \tilde{N}^{-2s_l}+  \tilde{N}/n,\nonumber\\
\sup_{f\in B_{2,\infty}^{s_l}(L)} E_{X_2|X_1}\|\hat{f}-f_0\|_2^21_{\mathcal{E}^c}&\leq  P_{X_2}(\mathcal{E}^c ) 2^{\tilde{j}_n+1}(n+L^2)=o(n^{-1}),\nonumber
\end{align}
for $l=1,2$.
Therefore, in view of assertion \eqref{eq: UB:tildeN}
\begin{align*}
\sup_{f\in B_{2,\infty}^{s_2}(L)}  E_{X_1,X_2}\|\hat{f}-f_0\|_2^2\lesssim \sup_{f\in B_{2,\infty}^{s_2}(L)} E_{X_1}\big( \tilde{N}^{-2s_2}+ \tilde{N}/n\big)\lesssim n^{-2s_2/(1+2s_2)},\\
\sup_{f\in R_{M_n^{-1}}^{s_1}(L)} E_{X_1,X_2}\|\hat{f}-f_0\|_2^2\lesssim \sup_{f\in R_{M_n^{-1}}^{s_1}(L)} E_{X_1}\big( \tilde{N}^{-2s_1}+ \tilde{N}/n\big)\lesssim n^{-2s_1/(1+2s_1)}.
\end{align*}

It remained to deal with the intermediate set, i.e. $ f_0\in B_{2,\infty}^{s_1}(L)\backslash R_{M_n^{-1}}^{s_1}(L)$. Our local estimator $\hat{s}_{n/m}^{(i)}$ will be either $s_1$ or $s_2$, hence for each machine the amount of transmitted bits is bounded from above by $ n^{1/(1+2\hat s_n^{(i)})}\log n\leq  n^{1/(1+2s_1)}\log n$ $P_{X^{(i,2)}}$-almost surely. Note that the median $\tilde{N}$ also satisfies almost surely that $n^{1/(1+2s_1)}\geq \tilde{N}\geq n^{1/(1+2s_2)}$. Then, using the notation $f_{0,j\leq \tilde{j}_n}=\sum_{j=0}^{\tilde{j}_n}f_{0,jk}\psi_{jk}$, we get similarly as above, that
\begin{align}
E_{X_1,X_2}\|\hat{f}-f_{0,j\leq \tilde{j}_n}\|_2^2&\leq E_{X_1}\sum_{j=0}^{\tilde{j}_n}\sum_{k=1}^{2^j}E_{X_2|X_1}( \frac{1}{\sqrt{n}}Z_{jk}+\eps_{jk})^2+o(n^{-1})\nonumber\\
&\lesssim E_{X_1}\tilde{N}/n\leq n^{-\frac{2s_1}{1+2s_1}}.\label{eq: help002}
\end{align}
To deal with the bias term let us denote by $\tilde{f}\in B_{2,\infty}^{s_2}(L)$ a function satisfying $\|f_0-\tilde{f} \|_2^2\lesssim \tilde{C}_{M_n^{-1}} (n/m)^{-2s_1/(1/2+2s_1)}$, then by recalling that $$(n/m)^{1/(1/2+2s_1)}= n^{(1-p)/(1/2+2s_1)}=n^{\frac{1/2-p(1+2s_1)}{(1+2s_1)(1/2+2s_1)}}  n^{1/(1+2s_1)},$$
we get that
\begin{align}\label{eq: help001}
E_{X_1}\|&f_{0,j\leq \tilde{j}_n}-f_0\|_2^2\leq E_{X_1}\sum_{j= \tilde{j}_{n}}^{\infty}\sum_{k=1}^{2^j}f_{0,jk}^2\\
&\leq 2E_{X_1}\Big(\sum_{j=\tilde{j}_{n}}^{\infty}\sum_{k=1}^{2^j}(f_{0,jk}-\tilde f_{jk})^2+\sup_{j\geq \tilde{j}_n}(2^{2js_2}\sum_{k=1}^{2^j}\tilde f_{jk}^2) \sum_{j=\tilde{j}_{n}}^{\infty}2^{-2js_2}\Big)  \nonumber\\
&\lesssim   \tilde{C}_{M_n^{-1}}^2 (n/m)^{-\frac{2s_1}{1/2+2s_1}}+E_{X_1}\tilde{N}^{-2s_2}\lesssim n^{-\frac{2s_1}{1+2s_1}},\nonumber
\end{align}
where the last inequality follows from $\tilde{C}_{M_n^{-1}}\asymp n^{\frac{s_1(1/2-p(1+2s_1))}{(1+2s_1)(1/2+2s_1)}}$. Then by combining \eqref{eq: help002} and \eqref{eq: help001} we get that $E_{X_1,X_2}\|\hat{f}-f_0\|_2^2\lesssim n^{-\frac{2s_1}{1+2s_1}}$, concluding the proof of the theorem.

\section{Proof of Corollary \ref{thm: adaptL2positive2}}\label{sec: adaptL2positive2}
We adapt the method and proof of Theorem \ref{thm: adaptL2positive} to the collection of regularity classes $s_0\in[s_1,s_2]$, where $s_0$ denotes the regularity of the truth we want to adapt to. Similarly to the discrete case we divide the data in each machine to two independent samples $X^{(i,1)}$ and $X^{(i,2)}$. Let $\mathcal{S}_n$ denote a $1/\log n$-grid of the interval $[s_1,s_2]$, i.e. $\mathcal{S}_n=\{s_1, s_1+1/\log n,..., s_2\}$, and denote by $\underline{s}=s_1+\gamma_n/\log n$, for some   $0\leq \gamma_n\leq \lceil(s_2-s_1)\log n\rceil$, $\gamma_{n}\in\mathbb{N}$, the lower bound of the $1/\log n$-bin containing $s_0$, i.e.  $s_0\in[\underline{s},\underline{s}+1/\log n]$. We will describe next a testing procedure for the regularity hyper-parameter $s_0$. Let us compute the test $\Psi_{n/m}^{(i)}(M_{n,t}^{-1},t,s)$ for all $t<s$, $s,t\in\mathcal{S}_n$ and take $\hat{s}_{n/m}^{(i)}$ to be the largest regularity $s$ for which the null hypothesis was retained for every $t<s$, i.e.
\begin{align*}
\hat{s}_{n/m}^{(i)}=\max\{s\in\mathcal{S}_n:\, \Psi_{n/m}^{(i)}(M_{n,t}^{-1},t,s)= 0,\,\,\forall t<s \}.
\end{align*}
The aggregated regularity estimator $\hat{s}$ and the distributed estimator $\hat{f}$ is then constructed the same way as in the proof of Theorem \ref{sec: adaptL2positive}, using the above defined $\hat{s}_{n/m}^{(i)}$.

The probability of under smoothing is bounded from above by $(\gamma_n-1)^2\leq (s_2-s_1)^2\log^2 n$ times the probability of rejecting the correct null-hypothesis. Hence in view of assertion \eqref{eq:cons:test} and the monotone decreasing property of the function $s\mapsto M_{n,s}$, we get that
\begin{align*}
P\Big(\hat{s}_{n/m}^{(i)}<\underline{s}\Big)\lesssim (s_2-s_1)^2(\log n)^2 e^{-M_{n,s_2}^{1/2}/2} =o(1).
\end{align*}

This implies that for all $i\in\{1,...,m\}$
\begin{align*}
E_{X^{(i,1)}, X^{(i,2)}}\hat{B}^{(i)}&=E_{X^{(i,1)}}\hat{B}^{(i)}\leq E_{X^{(i,1)}} n^{\frac{1}{1+2\hat{s}_{n/m}^{(i)}}}\log n\\
&\lesssim n^{\frac{1}{1+2\underline{s}}}\log n+n^{\frac{1}{1+2s_1}}e^{-M_{n,s_2}^{1/2}/2}\log^2 n\lesssim n^{\frac{1}{1+2s_0}}\log n
\end{align*}
 and similarly to assertions \eqref{eq: help:UB:inconsistency} and \eqref{eq: UB:tildeN} that
\begin{align}
&P_{X_1}(\hat{s}<\underline{s})=P_{X_1}\big(\tilde{N}>n^{\frac{1}{1+2\underline{s}}} \big)\leq e^{-m/5}\quad\text{and}\nonumber\\
&E_{X_1}\tilde{N}< n^{\frac{1}{1+2\underline{s}}}+ n^{\frac{1}{1+2s_1}} P_{X_1}\big(\tilde{N}>n^{\frac{1}{1+2\underline{s}}} \big)\lesssim n^{\frac{1}{1+2\underline{s}}}\lesssim n^{\frac{1}{1+2s_0}},\label{eq: help003}
\end{align}
for $m\geq 5\log n$.

It remaines to show that our procedure adapts to the minimax risk. First note that in view of assertion \eqref{eq: help002}  and \eqref{eq: help003}
\begin{align*}
\sup_{f_0\in B_{2,\infty}^{\underline{s}}}E_{X_1}\big(E_{X_2|X_1}\|\hat{f}-f_{0,j\leq \tilde{j}_{n}}\|_2^2\big) \leq E_{X_1}\tilde{N}/n\lesssim n^{-\frac{2s_0}{1+2s_0}}.
\end{align*}
Next let $j_{n,s}= (1+2s)^{-1}\log n$, then for $\tilde{j}_n=\lfloor\log \tilde{N}\rfloor$
\begin{align}
E_{X_1}&(\| f_{0,j\leq \tilde{j}_n}-f_0\|_2^2)\label{eq: UB:risk}\\
&=\Big(\sum_{s<\underline{s},\, s\in\mathcal{S}_n}+\sum_{s=\underline{s}}^{\underline{s}} + \sum_{s>\underline{s},\, s\in\mathcal{S}_n}\Big)P_{X_1}(\hat{s}=s) E_{X_1}\big(\| f_{0,j\leq j_{n,s}}-f_0\|_2^2\big |\hat{s}=s\big)\nonumber\\
&=\Big(\sum_{s<\underline{s},\, s\in\mathcal{S}_n}+\sum_{s=\underline{s}}^{\underline{s}} + \sum_{s>\underline{s},\, s\in\mathcal{S}_n}\Big)P_{X_1}(\hat{s}=s) \sum_{j=j_{n,s}}^{\infty}\sum_{k=1}^{2^j}f_{0,jk}^2.\nonumber
\end{align}

We deal with the three terms on the right hand side separately. In view of assertion \eqref{eq: help003} and $\|f_0\|_2^2\leq L^2$ we have that
\begin{align*}
\sum_{s<\underline{s}} P_{X_1}(\hat{s}=s)\sum_{j=j_{n,s}}^{\infty}\sum_{k=1}^{2^j}f_{0,jk}^2\leq L^2e^{-m/5}\lesssim n^{-\frac{2s_0}{1+2s_0}}.
\end{align*}
Then it is also easy to see that
\begin{align*}
 P_{X_1}(\hat{s}=\underline{s})\sum_{j=j_{n,\underline{s}}}^{\infty}\sum_{k=1}^{2^j}f_{0,jk}^2
&<  \sum_{j=j_{n,\underline{s}}}^{\infty}  2^{-2j\underline{s}} \sup_{j\geq j_{n,\underline{s}}}2^{2j\underline{s}}\sum_{k=1}^{2^j}f_{0,jk}^2\\
&\leq L^2 n^{-\frac{2\underline{s}}{1+2\underline{s}}}\lesssim n^{-\frac{2s_0}{1+2s_0}}.
\end{align*}

Then for arbitrary $s>\underline{s}$, $s\in\mathcal{S}_n$, using the notation $R_{M_{n,\underline{s}}^{-1}}^{\underline{s},s}(L):=\{f\in B_{2,\infty}^{\underline{s}}(L):\, \|f- B_{2,\infty}^{s}(L)\|_2\geq \tilde{C}_{M_{n,\underline{s}}^{-1}} (n/m)^{-\frac{\underline{s}}{1/2+2\underline{s}}}\}$,  we have that
\begin{align*}
 \sup_{f_0\in R_{M_{n,\underline{s}}^{-1}}^{\underline{s},s}(L)} P_{X^{(i,1)}}\big( \hat{s}_{n/m}^{(i)}\geq s  \big)
&\leq  \sup_{f_0\in R_{M_{n,\underline{s}}^{-1}}^{\underline{s},s}(L)} E_{X^{(i,1)}}\Big(1-\Psi_{n/m}^{(i)}\big(M_{n,\underline{s}}^{-1},\underline{s},s \big) \Big)\\
& \lesssim e^{-M_{n,\underline{s}}^{1/2}/2}.
\end{align*}
Therefore, by Hoeffding's inequality,
\begin{align}
 \sup_{f_0\in R_{M_{n,\underline{s}}^{-1}}^{\underline{s},s}(L)} P_{X_1}\big( \hat{s}\geq s   \big)\leq e^{-m/5},\label{eq: Hoeffding}
\end{align}
hence by combining the preceding two displays we get that
\begin{align*}
 \sup_{f_0\in R_{M_{n,\underline{s}}^{-1}}^{\underline{s},s}(L)}& \sum_{j=j_{n,s}}^{\infty}\sum_{k=1}^{2^j}f_{0,jk}^2 P_{X_1}(\hat{s}=s)\leq L^2e^{-m/5}=o(  n^{-2s_0/(1+2s_0)}/\log n).
\end{align*}
For any $f_0\in \mathcal{F}_s:= B_{2,\infty}^{\underline{s}}(L)\backslash R_{M_{n,\underline{s}}^{-1}}^{\underline{s},s}(L)$ there exists an $\tilde{f}_0\in B_{2,\infty}^{s}(L)$ such that $\|f_0-\tilde{f} \|_2\leq \tilde{C}_{M_{n,\underline{s}}^{-1}} (n/m)^{-\frac{\underline{s}}{1/2+2\underline{s}}}$.  Then similarly to assertion \eqref{eq: help001} we get that
\begin{align*}
  \sup_{f_0\in \mathcal{F}_s} \sum_{j=j_{n,s}}^{\infty}\sum_{k=1}^{2^j}f_{0,jk}^2
&\leq 2 \sup_{f_0\in \mathcal{F}_s} \Big(\sum_{j=j_{n,s}}^{\infty}\sum_{k=1}^{2^j}(f_{0,jk}-\tilde f_{0,jk})^2+\sum_{j=j_{n,s}}^{\infty}2^{-2js} \sup_{j\geq j_{n,s}} 2^{2js}\sum_{k=1}^{2^j}\tilde f_{0,jk}^2 \Big)  \nonumber\\
&\lesssim   \tilde{C}_{M_{n,\underline{s}}^{-1}} (n/m)^{-\frac{2\underline{s}}{1/2+2\underline{s}}}+ 2^{-2j_{n,s}s}\\
&\lesssim n^{-\frac{2\underline{s}}{1+2\underline{s}}}+ n^{-\frac{2s}{1+2s}}\lesssim n^{-\frac{2s_0}{1+2s_0}}.
\end{align*}
Hence 
\begin{align*}
\sup_{f_0\in B_{2,\infty}^{\underline{s}}(L)}&\sum_{s>\underline{s}}^{s_2} P_{X_1}(\hat{s}=s)\sum_{j=j_{n,s}}^{\infty}\sum_{k=1}^{2^j}f_{0,jk}^2\\
&\lesssim \sum_{s>\underline{s}}^{s_2} \big(P_{X_1}(\hat{s}=s)+o(1/\log n)\big) n^{-\frac{2s_0}{1+2s_0}}\lesssim n^{-\frac{2s_0}{1+2s_0}}.
\end{align*}

Combining the upper bounds above  we get that
\begin{align*}
\sup_{f_0\in B_{2,\infty}^{\underline{s}}(L)} E_{X_1,X_2}\|\hat{f}-f_0\|_2^2&\leq 2
\sup_{f_0\in B_{2,\infty}^{\underline{s}}(L)}\Big( E_{X_1}\| f_{0,j\leq \tilde{j}_{n}}-f_0\|_2^2\\
&\qquad\qquad+ 
E_{X_1,X_2}\|\hat{f}-f_{0,j\leq \tilde{j}_{n}}\|_2^2\Big)\\
&\lesssim n^{-\frac{2s_0}{1+2s_0}},
\end{align*}
concluding the proof of the corollary.

\subsection{Proof of Theorem \ref{thm: adaptLinftynegative}}\label{sec: adaptLinftynegative}
The proof follows the same lines of reasoning as  the proof of Theorem \ref{thm: adaptL2negative}, here we highlight only the differences. 

First of all the set of functions $\mathcal{F}_0$ is defined slightly differently. Let us introduce the notations
\begin{align}
\tilde\delta_n&=\bar\delta_n\wedge (m/n) ,\quad \text{with}\label{def: tilde:delta_n}\\
\bar{\delta}_n&=\min\Big\{\frac{m}{n \log m }, \frac{1}{n [\bar\delta_n^{1/(1+2s_1)} \beta_n \wedge  1]  \log m} \Big\},\nonumber\\
\beta_n&=(\Gamma_n \vee n^{\frac{1}{1+2s_1}-\eps_1})\log n \quad\text{and}\quad \Gamma_n=  n^{\frac{1/2}{1+2s_1}+\frac{1/2}{1+2s_2}+\eps_1},
\nonumber
\end{align}
with $\eps_1\in (0,\frac{s_2-s_1}{(1+2s_1)(1+2s_2)}\wedge \frac{(1-p)/8}{1+2s_1})$. By elementary computations one can deduce that $\bar{\delta}_n\geq n^{\eps_1/2-1}$ and therefore 
\begin{align}
\tilde{\delta}_n\geq n^{(\eps_1/2\wedge p)-1}.\label{eq:LB:tilde:delta:Linfty}
\end{align}
Next, let us denote by ${K}_j$ the largest set of Daubechies wavelets with disjoint supports at resolution level $j$. Note that $|{K}_j|\geq c_0 2^j$ (for large enough $j$ and sufficiently small $c_0>0$).
Then we consider the class of functions
\begin{align}
 \mathcal{F}_0=\{f_k:\, k\in {K}_{j_n}\} ,\quad \text{where}\quad f_{k}=\tilde\delta_n^{1/2} \psi_{j_n,k}. \label{eq: funcLinf2}
\end{align}
Since the functions in $ \mathcal{F}_0$ have disjoint supports we have
\begin{align*}
\sup_{f\in\mathcal{F}_0}\|0-f\|_\infty&= \sup_{k\in K_{j_n}}  \tilde\delta_n^{1/2}  \|\psi_{j_n,k}\|_{\infty}\lesssim  2^{j_n/2}\tilde\delta_n^{1/2}\\
&\lesssim \tilde\delta_n^{s_1/(1+2s_1)} \ll (n/m)^{-s_1/(1+2s_1)},
\end{align*}
following from the definition of $\tilde\delta_n$. Hence it is not possible to test between the zero function and the set $ \mathcal{F}_0$ in the local servers. 

Using the notation $Z$ for the likelihood ratio introduced in the proof of Theorem \ref{thm: adaptL2negative}
we note that in view of the proof of Theorem 6.2.11 b) on page 493 of \cite{gine:nickl:2016} we have that
\begin{align*}
E(Z-1)^2\leq (e^{\bar\gamma_n^2}-1)/| \mathcal{F}_0|, \quad\text{where}\quad \bar\gamma_n=\sqrt{\tilde\delta_n n/m}.
\end{align*}
Then the infimum of the tests given in  \eqref{eq: test:nonconsistent} is bounded from below by $(1-C\eta_n)^2$ for $\eta_n=\tilde\delta_n^{1/(4+8s_1)}\leq n^{-(1-p)/(4+8s_1)}\leq n^{-2\eps_1}$. This leads to
\begin{align*}
\frac{1}{| \mathcal{F}_0|}\sum_{f\in \mathcal{F}_0} E_f^{(i)} \hat{B}^{(i)}\lesssim  \Gamma_n+ n^{\frac{1}{1+2s_1}+\eps_1}(\log n)( \eta_n+ n^{\frac{1/2}{1+2s_2}-\frac{1/2}{1+2s_1}})\lesssim \beta_n.
\end{align*}

This means that the expected number (with respect to the joint distribution of the variables $F$ and $P_f$, $f\in\mathcal{F}_0$) of transmitted bits on the class $\mathcal{F}_0$ is bounded from above by a multiple of $\beta_n$. So the distributed estimator satisfies assertion \eqref{eq: UB:info:length} in with $B^{(i)}$ replaced by $C\beta_n$. Hence in view of the minimax lower bound derived in assertion \eqref{eq: LB:average:bits:infty} (with $B^{(i)}$ replaced by $\beta_n$ in the definition of $\delta_n$ in the proof of Theorem \ref{thm: minimaxLinftyLBgeneral}) and the definition of $\tilde{\delta}_n$
\begin{align*}
\sup_{f_0\in \mathcal{F}_0} E_{f_0}\|\hat{f}-f_0\|_\infty \gtrsim \tilde{\delta}_n^{\frac{s_1}{1+2s_1}}\gg n^{-\frac{s_1}{1+2s_1}+\eps_2},
\end{align*}
with $\eps_2=(\eps_1/2\wedge p)s_1/(1+2s_1)$, where the last inequality followed from \eqref{eq:LB:tilde:delta:Linfty}. 
This contradicts assumption \eqref{eq:UB:risk}, finishing the proof of our statement.

\subsection{Proof of Theorem \ref{thm: adaptLinftypositive}}\label{sec: adaptLinftypositive}
First note that in Lemma 5.2 of \cite{bull:2012} it was shown that the smoothness can be consistently estimated under the self-similarity condition, i.e. there exists an estimator $\hat{s}_{n/m^{(i)}}$  such that for every $i\in\{1,...,m\}$ and $c>0$ there exists $C>0$ satisfying
\begin{align}
\inf_{s\in[s_1,s_2]}\inf_{f_0\in S^s_{\infty}(L,\eps,j_0)}P_{f_0}(s- C/\log (n/m) \leq \hat{s}_{n/m}^{(i)}\leq s)\lesssim (m/n)^{c}.\label{eq: estimator_s}
\end{align}
By choosing $c=1/(1-p)$ we have $(m/n)^{c}=1/n$. Then we propose a similar estimation method as in Theorem \ref{thm: adaptL2positive}. First we split the data into $X^{(i,1)}$ and $X^{(i,2)}$ and use the first sample $X^{(i,1)}$ to construct the estimator $\hat{s}_{n/m}^{(i)}$ for the smoothness parameter $s$. Next transmit the approximation of the first  $\tilde{N}^{(i)}=(n/\log n)^{1/(1+2\hat{s}_{n/m}^{(i)})}$ coefficients (instead of $n^{1/(1+2\hat{s}_{n/m}^{(i)})}$ as in Theorem  \ref{thm: adaptL2positive}) of the second subset of observations $X^{(i,2)}$, following Algorithm \ref{alg: transmit:number}. Then $\hat{B}^{(i)}\leq (n/\log n)^{1/(1+2s_1)}\log n$ and
\begin{align*}
E_{X^{(i,1)},X^{(i,2)}} \hat{B}^{(i)}&=E_{X^{(i,1)}}\hat{B}^{(i)}= E_{X^{(i,1)}}\tilde{N}^{(i)}\log n\\
&\leq (n/\log n)^{\frac{1}{1+2s}}\log n+n^{-1} (n/\log n)^{\frac{1}{1+2s_1}}\log n\\
&\lesssim n^{\frac{1}{1+2s}} (\log n)^{\frac{2s}{1+2s}}.
\end{align*}
Besides we also have that the median $\tilde{N}$ of the values  $\tilde{N}^{(i)}$ satisfy that
\begin{align}
P_{X_{1}}(n^{1/(1+2s)}\leq \tilde{N} \leq C_1n^{1/(1+2s)})\geq 1-C_2e^{-m/5},\label{eq: bounds:median}
\end{align}
for some large enough constants $C_1,C_2>0$.

Similarly to before let $\tilde{j}_n=\lfloor \log \tilde{N}\rfloor$ and $f_{0,j\leq\tilde{j}_n}=\sum_{j\leq \tilde{j}_n}\sum_{k=1}^{2^j}f_{0,jk}\psi_{jk}$. Then using the notation $\mathcal{E}$ introduced in \eqref{def: setE} we get that
\begin{align*}
\|\hat{f}-f_0\|_{\infty}1_{\mathcal{E}}&\leq \|\hat{f}-f_{0,j\leq \tilde{j}_n}\|_{\infty}1_{\mathcal{E}}+\|f_{0,j\leq \tilde{j}_n}-f_0\|_{\infty}\\
&\leq \|\sum_{ j\leq \tilde{j}_n} \sum_{k=1}^{2^j}\frac{1}{|M_{jk}|}\sum_{i\in M_{jk}}(\sqrt{\frac{m}{n}}Z_{jk}^{(i)}+\eps_{jk}^{(i)})\psi_{jk}\|_{\infty}1_{\mathcal{E}}+\sum_{j= \tilde{j}_n}^{\infty}2^{j/2}\sup_{k\in K_j}  |f_{0,jk}| \\
&\lesssim \sup_{j\leq  \tilde{j}_n}\Big( \Big|\frac{1}{|M_{jk}|}\sum_{i\in M_{jk}}\sqrt{\frac{m}{n}}Z_{jk}^{(i)}\Big|+n^{-1/2}\Big)\sum_{j=0}^ {\tilde{j}_n} 2^{j/2}  +\sum_{j= \tilde{j}_n}^{\infty}2^{j/2}\sup_{k\in K_j}  |f_{0,jk}|\\
&\lesssim \sqrt{\frac{\tilde{N}}{n}}\sup_{j\in\{1,...,  \tilde{j}_n\}}\sup_{k\in K_j} (|Z_{j,k}|+1)+2^{-  \tilde{j}_ns}\sum_{j= \tilde{j}_n}^{\infty}2^{j (s+1/2)}\sup_{k\in K_j}  |f_{0,jk}|,
\end{align*}
where $Z_{jk}:= \frac{\sqrt{n}}{|M_{jk}|}\sum_{i\in M_{jk}}\sqrt{\frac{m}{n}}Z_{jk}^{(i)}\stackrel{iid}{\sim} N(0,\frac{m}{ |M_{jk}|})$, $0\leq \eps_{jk}^{(i)}\leq 1/\sqrt{n}$ on $\mathcal{E}$.
Therefore in view of \eqref{eq: bounds:median}
\begin{align*}
E_{X_1,X_2}\|\hat{f}-f_0\|_\infty&\lesssim E_{X_1} \sqrt{\frac{\tilde{N}}{n}}\log \tilde{N}+E_{X_1}\tilde{N}^{-s}+o(n^{-1})\\
&\lesssim (n/\log n)^{-\frac{s}{1+2s}}+e^{-m/5}\lesssim (n/\log n)^{-\frac{s}{1+2s}}.
\end{align*}
concluding the proof of our statement.

\section{Technical lemmas}

The first lemma extends sligthly the results of Shannon's source coding theorem by allowing also non-prefix codes, see Lemma 5.1 of \cite{szabo:zanten:2018}.
\begin{lemma}\label{lem: Shannon}
Let $Y$ be a random finite binary string. 
Its expected length  satisfies the inequality
\begin{align*}
H(Y)\leq 2 \EE l(Y)+1.
\end{align*}
\end{lemma}

 Let us take an arbitrary $x\in\mathbb{R}$ and write it in a scientific binary representation, i.e. $|x|=\sum_{k=-\infty}^{\log_2|x|} b_k2^k$, with $b_k\in\{0,1\}$, $k\in \mathbb{Z}$. Then let us take $y$ consisting the same digits as $x$ up to the $(D\log_2 n)th$ digits, for some $D>0$, after the binary dot (and truncated there), i.e. $|y|=\sum_{k=-D\log_2 n}^{\log_2|x|} b_k2^k$, unless $|x|\geq\sqrt{n}$, in which case we set $y$ to zero, see also Algorithm \ref{alg: transmit:number}, a slightly modified version of Algorithm 1 from \cite{szabo:zanten:2018}. In the algorithm the function $x\mapsto\sign(x)$ is one if $x\geq 0$ and zero otherwise.

\begin{algorithm}
\caption{Transmitting a finite-bit approximation of a number}\label{alg: transmit:number}
\begin{algorithmic}[1]
\Procedure{TransApprox($x$)}{}
\If{$|x|\geq n $}
\State \textit{Transmit:} $\sign(x)$, $b_{-\lfloor D\log n\rfloor+1},...,b_{\lfloor\log|x|\rfloor} $.
\State \textit{Construct:} $y=(2\sign(x)-1)\sum_{k= -D\log n+1}^{\log|x|}b_k2^{k}$.
\Else \State \textit{Transmit:} 0.
\State\textit{Construct:} $y=0$. 
\EndIf
\EndProcedure
\end{algorithmic}
\end{algorithm}

The next lemma gives an upper bound for the number of transmitted bits  and the accuracy of the procedure described in Algorithm \ref{alg: transmit:number}. It is a slightly reformulated version of Lemma 2.3 of \cite{szabo:zanten:2018} to accommodate almost sure upper bound on the code length.

\begin{lemma}\label{lem: approx}
For $X\sim N(\mu,\sigma^2)$, with $|\mu|\leq M$ and $\sigma\leq 1$ let the approximation $Y$ of $X$ given in Algorithm \ref{alg: transmit:number} and denote by $\mathcal{E}_X$ the event that $|X|\leq\sqrt{n}$. Then for large enough $n$,
\begin{align*}
P_X(\mathcal{E}_X^c)=O(e^{-cn}) \quad, |X-Y|1_{\mathcal{E}_X}<2 n^{-D}, \text{and}\quad
 l(Y)\leq (D+1/2)\log n,
\end{align*}
for some $c>0$.
%Furthermore we also have that
%$$\PP\big( l(Y)\leq 1+(D+1)\log n\big)\geq 1-cn^{-3/2},$$
%for some large enough constant $c>0$.
\end{lemma}

 \begin{proof}

It is straightforward to see that  the last two inequalities of the statement hold.  To prove the first one note that

\begin{align*}
P_X(\mathcal{E}_X^c)\leq P_X(|X|\geq \sqrt{n})&\leq P_X( |X-\mu|\geq \sqrt{n}-M)\lesssim e^{-cn}.
\end{align*}

\end{proof}

Next we provide an extended version of Lemma 4.2 of \cite{carpentier:2015} with tighter upper bounds for small $\Delta>0$. The main difference in the proof is that instead of Chebyshev's inequality we apply a more accurate concentration inequality, see Lemma 8.1 of \cite{birge:2001}.

\begin{lemma}\label{lem: redo:Lem4.2}
Let $\Delta>0$. Then
\begin{align*}
P\Big\{\forall l: J_0\leq l\leq j, |T_n(l)-\| \Pi_l f\|_2^2|\geq 4 \sqrt{\frac{3z_0}
{\Delta} \Big( \frac{2^{(j+l)/2}}{n^2}+2^{l/4}\frac{\|\Pi_lf\|_2^2}{n} \Big) } \Big\}\leq 2e^{-c/\sqrt{\Delta}},
\end{align*}
for $c=\sqrt{3/2}$ and $z_0=2^{J_0}$ the number of father wavelets (at resolution level $J_0$) and $\Pi_l f=\sum_{k=1}^{2^l} f_{lk}\psi_{lk}$ the projection of $f$ into the wavelet resolution level $l$.
\end{lemma}

\begin{proof}
Note that for the wavelet estimator $\hat{f}$ with signal-to-noise ration $n$we get that $\|\Pi_l\hat{f}\|_2^2=\sum_k \hat{f}_{lk}^2$, where $\hat{f}_{lk}-f_{lk}\stackrel{iid}{\sim}N(0,1/n)$.

Hence in view of Lemma 8.1 of \cite{birge:2001} (with degree of freedom $D=2^l$, non-centrality parameter $B=n\sum_{k=1}^{2^l} f_{lk}^2$ and $x=1/(2\sqrt{\delta_l})$)  we get for $\delta_l\leq 1/4$ that

\begin{align*}
&P\Big\{ \Big|  \|\Pi_l\hat{f}\|_2^2-\frac{2^l}{n}-\|\Pi_lf\|_2^2 \Big|
\geq \sqrt{\frac{4}{\delta_l} \Big(\frac{2^l}{n^2}+\frac{\|\Pi_lf\|_2^2}{n}\Big) }  \Big\}\\
&\qquad=P\Big\{ \Big|  \sum_{k=1}^{2^l}\hat{f}_{lk}^2-\frac{2^l}{n}-\sum_{k=1}^{2^l}f_{lk}^2 \Big|
\geq \sqrt{\frac{4}{\delta_l} \Big(\frac{2^l}{n^2}+\frac{\sum_{k=1}^{2^l}f_{lk}^2}{n}\Big) }  \Big\}\\
&\qquad \leq
P\Big\{ \Big|  \sum_{k=1}^{2^l}n\hat{f}_{lk}^2-2^l-n\sum_{k=1}^{2^l}f_{lk}^2 \Big|
\geq 2\sqrt{ \Big(2^l+ 2n\sum_{k=1}^{2^l}f_{lk}^2\frac{1}{2\sqrt{\delta_l}}\Big)}+2 \frac{1}{2\sqrt{\delta_l}}  \Big\}\\
&\qquad\leq 2 e^{-0.5/\sqrt{\delta_l}}.
\end{align*}
Similarly
\begin{align*}
P\Big\{ \Big|  \|\Pi_{J_0}\hat{f}\|_2^2-\frac{z_0}{n}-\|\Pi_{J_0}f\|_2^2 \Big|
\geq \sqrt{\frac{4}{\delta_{J_0}} \Big(\frac{z_0}{n^2}+\frac{\|\Pi_{J_0}f\|_2^2}{n}\Big) }  \Big\}
\leq 2 e^{-0.5/\sqrt{\delta_{J_0}}}.
\end{align*}
By the definition of $T_n(l)$ and union bound these results imply that
\begin{align*}
P\Big\{  \forall l:J_0<l\leq j, \big|T_n(l)- \|\Pi_l f\|_2^2\big|\geq \sqrt{\frac{4}{\delta_l} \Big(\frac{2^l}{n^2}+\frac{\|\Pi_lf\|_2^2}{n}\Big) },\\
\big|T_n(J_0)- \|\Pi_{J_0}f\|_2^2\big|\geq  \sqrt{\frac{4}{\delta_{J_0}} \Big(\frac{z_0}{n^2}+\frac{\|\Pi_{J_0}f\|_2^2}{n}\Big) }  \Big\}\leq \sum_{J_0\leq l\leq j}e^{-0.5/\sqrt{\delta_l}}.
\end{align*}
Setting similarly to Lemma 4.2 of \cite{carpentier:2015} the parameters $\delta_l=(2^{-(j-l)/2}+2^{-l/4})\Delta/12$ and $\delta_{J_0}=\Delta/12$ we get in view of
\begin{align*}
\sum_{l=J_0}^{j} e^{-0.5/\sqrt{\delta_j}}\leq \sum_{l=J_0}^{j} \Big(e^{-\sqrt{3/2}\Delta^{-1/2}2^{(j-l)/4}}+ e^{-\sqrt{3/2}\Delta^{-1/2}2^{l/8}}\Big)\lesssim e^{-\sqrt{3/2}\Delta^{-1/2}}
\end{align*}
which implies together with $z_0\geq1$ that
\begin{align*}
P\Big\{  \forall l:J_0\leq l\leq j, |T_n(l)- \|\Pi_l f\|_2^2|&\geq 4\sqrt{\frac{3z_0}{\Delta} \Big(\frac{2^{(j+l)/2}}{n^2}+2^{l/4}\frac{\|\Pi_lf\|_2^2}{n}\Big) }\\
&\lesssim  e^{-\sqrt{3/2}\Delta^{-1/2}},
\end{align*}
concluding the proof of the lemma.
\end{proof}

The next lemma is a slightly rewritten version of Theorem 3.1 of \cite{carpentier:2015} with tighter error bounds (for small $\alpha>0$).

\begin{lemma}\label{lem: redo:carpentier}
Let $\alpha>0$. The test $\Psi_n(\alpha)$ satisfies that for all $\alpha>0$ and $n>0$
\begin{align*}
\sup_{f\in H_0}E_f \Psi_n +\sup_{f\in H_1}E_f (1-\Psi_n)\leq 2e^{-1/\sqrt{\alpha}},
\end{align*}
where 
\begin{align*}
H_0: f\in B_{2,\infty}^{s_2}(L)\quad\text{and}\quad H_1: f\in  \{B_{2,\infty}^{s_1}(L): \|f-B_{2,\infty}^{s_2}(L) \|_2\geq \rho_n\}, 
\end{align*}
with $\rho_n=\tilde{C}_{\alpha} n^{-s_1/(1/2+2s_1)}$ and $\tilde{C}_{\alpha}=24\big(\frac{2^{s_1}L}{\sqrt{1-2^{-2s_1}}+19}\sqrt{1/\alpha}\big)$.
\end{lemma}

\begin{proof}
The proof goes the same way as of Theorem 3.1 of \cite{carpentier:2015}, with the only difference that we apply Lemma \ref{lem: redo:Lem4.2} instead of Lemma 4.2 of \cite{carpentier:2015}. 
\end{proof}

We also recall a slight modification of Fano's inequality, see Corollary 1 of \cite{duchi:wainwright:2013} or Theorem A.6. of \cite{szabo:zanten:2018}.
Given a finite set $\FF_0\subset \mathcal{F}$, we use the notations
\begin{align*}
N_{t}^{\max}= \max_{f\in\mathcal{F}_0}\Big\{\#\{ \tilde{f}\in\mathcal{F}_0:\, d(f,\tilde{f})\leq t\} \Big\},\\
N_{t}^{\min}= \min_{f\in\mathcal{F}_0}\Big\{\#\{ \tilde{f}\in\mathcal{F}_0:\, d(f,\tilde{f})\leq t\} \Big\}.
\end{align*}

\begin{theorem}\label{lem: Fano:Wainwright}
If $\mathcal{F}$ contains a finite set $\mathcal{F}_0$ { 
 and} $|\mathcal{F}_0|-N_{t}^{\min}>N_t^{\max}$, {then} for all $p, t > 0$,
\begin{align*}
\inf_{\hat f \in \EEE(Y)} \sup_{f \in \FF}
\mathbb{E}_{f} d^p(\hat f, f) \ge t^p 
\Big(1-\frac{I(F;Y)+\log 2}{\log (|\mathcal{F}_0|/N_{t}^{\max})}\Big),
\end{align*} 
where $ \EEE(Y)$ denotes the set of all estimators depending only on $Y$ and the function class $\mathcal{F}$, and $F$ is a uniformly distributed random variable on $\mathcal{F}_0$.
\end{theorem}

The next lemma gives an upper bound for the mutual information between the uniform random variable $F$ on $\mathcal{F}_0\subset \mathbb{R}^d$ and the set of observations on all local machines $Y=(Y^{(1)},...,Y^{(m)})$ in the $d$-dimensional many normal means model.

\begin{lemma}\label{lem: lem5_redo}
Let  $F=\delta \beta$, with {$\delta^2\leq 2^{-10}m/(n \log (md))$} and $\beta$ a uniformly distributed random variable over $\{-1,1\}^d$. Furthermore, suppose that  $X=(X^{(1)},...,X^{(m)})$, where $X^{(i)}$s are $d$-dimensional random variables satisfying that {$X_j^{(i)} \given F_j$ and $F_j$ are independent of $F_{-j}$}, and  $X_j^{(i)} \given (F=f)\sim \mathbb{P}_{f_j}^{(i)}=N(f_j,m/n)$. Then
\begin{align*}
I(F;Y)\leq  \sum_{i=1}^{m}\frac{2\delta^2}{m/n}\min\Big\{ 2^{10}\log(md)H(Y^{(i)}),d\Big\}+4\log 2, 
\end{align*} 
where $I(F;Y)$ is the mutual information between $F$ and $Y$ in the Markov chain $F\rightarrow X\rightarrow Y$.
\end{lemma}

\begin{proof}

Let us introduce the notation $a^2=2^4\log(md)m/n$ and note that
\begin{align*}
\sup_{|x|\leq a} \frac{\phi_{\delta,m/n}(x)}{\phi_{-\delta,m/n}(x)}\leq
\sup_{|x|\leq a}e^{\frac{n|(x-\delta)^2-(x+\delta)^2|}{2m}}\leq \sup_{|x|\leq a}e^{\frac{2n\delta |x| }{m}}\leq e^{ \frac{2an\delta}{m}},
\end{align*}
where $\phi_{\mu,\sigma^2}$ denotes the density function of a normal distribution with mean $\mu$ and variance $\sigma^2$. Furthermore, let us introduce the notation $B_j=\{|x_j|\leq a\}$, $j=1,...,d$. Then by Theorem \ref{thm: mutualUB:independent} (with $\mathcal{F}_0=\big\{f=\delta\beta:\,\beta\in \{-1,1\}^d\big\}$) we have that
\begin{align}
I(F;Y^{(i)})&\leq d (\log 2) \sqrt{P_{X_j^{(i)}}(X_j^{(i)}\notin B_j)}+d^2  P_{X_j^{(i)}}(X_j^{(i)}\notin B_j)\nonumber\\
&\qquad\qquad+ 2{C^2(C-1)^2}  I(X^{(i)};Y^{(i)}),\label{eq: Ub:mutual:help}
\end{align}
with $C=e^{2^{3}|\delta|\sqrt{\log(md){n/m}} }$.
Next note that for $Z\sim N(0,m/n)$
\begin{align*}
P_{X_j^{(i)}}(X_j^{(i)}\notin B_j)\leq P( |Z| \geq a-\delta)\leq 2 e^{- \frac{(a-\delta)^2n}{2m}}
\leq 2 e^{- \frac{a^2n}{4m}}\leq 2 (md)^{-4},
\end{align*}
and the inequality $I(X^{(i)};Y^{(i)})\leq H(Y^{(i)})$ holds. Then by plugging in the above inequalities into \eqref{eq: Ub:mutual:help} and using the inequalities $e^{x}\leq 1+2x$ for $x\leq 0.4$ and ${C^2}\leq 2$  we get that
\begin{align*}
I(F;Y^{(i)})\leq  \sqrt{2}(\log 2)m^{-2}d^{-1}+ 2(\log 2)m^{-4} d^{-2} + 2^{11} \delta^2 {\frac{\log(md)n}{m}} H(Y^{(i)}).%\label{eq: Ub:mutual:help2}
\end{align*}

Furthermore, from the data-processing inequality and the convexity of the KL divergence 
\begin{align*}
 I(F;Y^{(i)})&\leq  I(F;X^{(i)})
\leq \frac{1}{|\mathcal{F}_0|^2}\sum_{f,f'\in\mathcal{F}_0}K(\mathbb{P}_f^{(i)}\| \mathbb{P}_f^{(i)})\\
&=\frac{\delta^2}{2m/n}\frac{1}{|\mathcal{F}_0|^2}\sum_{f,f'\in\mathcal{F}_0}\|\beta-\beta'\|_2^2
\leq 2(n/m)d\delta^2.
\end{align*}
We conclude our statement by noting that
\begin{align*}
I(F;Y) \leq \sum_{i=1}^m I(F;Y^{(i)})
\end{align*}
\end{proof}

The next theorem provide an upper bound for the mutual information, see Theorem A.9 in \cite{szabo:zanten:2018} or  Lemma 3 of \cite{zhang:2013}.

\begin{theorem}\label{thm: mutualUB:independent}
Let us consider the Markov chain $F\rightarrow X^{(i)}\rightarrow Y^{(i)}$, where $F$ is the
uniform distribution on $\FF_0\subset \mathbb{R}^d$ and $X^{(i)} \given (F=f) \sim P_{X^{(i)}|F=f}$ is a $d$-dimensional random variable. Assume that {$X_j^{(i)} \given F_j$ and $F_j$ are independent of $F_{-j}$}. For $C \ge 1$, define
\[
B_j = \Big\{ x_j: \max_{f \not =  f'} \frac{p(x_j \given f_j)}{p(x_j \given f'_j)} \le C\Big\}
\]
for a constant $C \ge 1$ and density $p(x_j|f_j)$. Then
\begin{align*}
I(F; Y^{(i)}) & \le  \sum_{j=0}^d \Big((\log 2) \sqrt{P_{X_j^{(i)}}(X_j^{(i)} \not\in B_j)}+
\log|\FF_0| P_{X_j^{(i)}} ({X_j^{(i)}}\not \in B_j)\Big)\\
&\qquad\qquad+ 2{C^2(C-1)^2} I(X^{(i)};Y^{(i)}),
\end{align*}
where $I(X^{(i)};Y^{(i)})$ is the mutual information between $X^{(i)}$ and $Y^{(i)}$. 
\end{theorem}

\appendix

\section{Proofs for the minimax rates in the Gaussian white noise model}
\subsection{Proof of Theorem \ref{thm: minimaxL2LB}}\label{sec: minimaxL2LB}

The proof of the theorem follows from the following, more general theorem with taking $B^{(1)}=...=B^{(m)}=B$. The proof is slight extension for a larger set of estimators and adaptation to the Gaussian white noise setting of the proof of Theorem 2.1 \cite{szabo:zanten:2018}.

\begin{theorem}\label{thm: minimaxL2LBgeneral}
Let the sequence $\delta_n=o(1)$ be defined as
\begin{align}
\delta_n=\min\Big\{\frac{m}{n \log n }, \frac{m}{n \sum_{i=1}^m [\delta_n^{\frac{1}{1+2s}} B^{(i)}\log n \wedge 1]  } \Big\}.\label{def: delta_n}
\end{align}
Then in the distributed Gaussian white noise model \eqref{model: GWN} we have for any $s>0$ that
\begin{align*}
\inf_{\hat{f}\in\mathcal{F}_{dist}(B^{(1)},...,B^{(m)})}\sup_{f_0\in B_{2,\infty}^s(L)} E_{f_0}\|\hat{f}-f_0\|_2^2 \gtrsim \delta_n^{\frac{2s}{1+2s}}.
\end{align*}
\end{theorem}

\textit{Proof of Theorem \ref{thm: minimaxL2LBgeneral}.}
Note that without loss of generality we can multiply $\delta_n$ with an arbitrary constant.
In the proof we define $\delta_n$ as the solution to
\begin{align}
\delta_n ={2^{-15}}L^{-2}\min\Big\{\frac{m}{n \log n }, \frac{m}{n \sum_{i=1}^m [\delta_n^{\frac{1}{1+2s}}\log (n) B^{(i)} \wedge  1]} \Big\}.\label{def: delta_n_new}
\end{align}
We note, however, that all the computations below hold for arbitrary $\delta_n'\leq \delta_n$ as well.

We prove the desired lower bound for the minimax risk using a 
modified version of Fano's inequality, given in Theorem \ref{lem: Fano:Wainwright}.
As a first step we construct a  finite subset $\FF_0 \subset B_{2,\infty}^s(L)$. 
We use the wavelet notation outlined in Appendix \ref{sec: wavelets} and
define $j_n=\lfloor(\log \delta_n^{-1})/(1+2s)\rfloor$. 
For $\beta \in\{-1,1\}^{2^{j_n}}$, let $f_\beta \in L_2[0,1]$ be 
the function with  wavelet coefficients  
 \begin{align}\label{eq: function_minimax}
 f_{\beta, jk}=
 \begin{cases}
L\beta_{k}\delta_n^{1/2}, & \text{if}\quad j=j_n,\, k=1,...,2^{j_n},\\
 0, & \text{else}. 
 \end{cases}
\end{align}
Now define $\FF_0 = \{f_\beta: \beta \in\{-1,1\}^{2^{j_n}}\}$.
 Note that $\FF_0 \subset B_{2,\infty}^s(L)$, since
\begin{align*}
\|f_\beta\|_{B_{{2},\infty}^s}^2=\sup_{j}2^{2sj}\sum_{k=1}^{2^j}f_{\beta, jk}^2 =L^22^{(2s+1)j_n}\delta_n \leq  L^2.
\end{align*}
Therefore, for an arbitrary set of estimators $\hat{\mathcal{F}}$ we have that
\begin{align*}
\inf_{\hat{f}\in \hat{\mathcal{F}}}\sup_{f_0\in B_{2,\infty}^{{s}}(L)}\mathbb{E}_{f_0}\|\hat{f}-f_0\|_2^2\geq \inf_{\hat{f}\in\hat{\mathcal{F}}}\sup_{f_0\in \mathcal{F}_0}\mathbb{E}_{f_0}\|\hat{f}-f_0\|_2^2.
\end{align*}
To prove the statement of the theorem we take the set of distributed estimators $\hat{\mathcal{F}}=\mathcal{F}_{dist}(B^{(1)},\ldots, B^{(m)};B_{2,\infty}^s(L))$, but the inequality holds more generally.

For this set of functions $\FF_0$, the maximum and minimum number of elements
in balls of radius $t >0$, given by  
\begin{align*}
N_{t}^{\max}= \max_{f_\beta\in\FF_0}\Big\{\#\{ f_{\beta'}\in\FF_0:\, 
\|f_\beta-f_{\beta'}\|_2\leq t\} \Big\},\\
N_{t}^{\min}= \min_{f_\beta\in\FF_0}\Big\{\#\{ f_{\beta'}\in\FF_0:\, 
\|f_\beta-f_{\beta'}\|_2\leq t\} \Big),
\end{align*}
satisfy $N_{t}^{\max}=N_t^{\min}$ and {$N_{t}^{\max}=\sum_{i=o}^{\tilde{t}}{2^{j_n} \choose i}<|\FF_0|/2$ for $\tilde{t}:=\frac{t^2}{4\delta_n L^2}<2^{j_n-1} $} (and therefore $N_{t}^{\max}<|\FF_0|-N_t^{\min}$).

Recall the notations $X = (X^{(1)}, \ldots, X^{(m)})$ for the data available at the local 
machines and $Y = (Y^{(1)}, \ldots, Y^{m)})$ for the binary messages transmitted to the 
central machine satisfying the distribution protocol, and consider the Markov chain
 $F \to X \to Y$, where $F$ is a uniform random 
element in $\FF_0$. 
It then follows from Theorem \ref{lem: Fano:Wainwright} (with {$t^2=L^2\delta_n2^{j_n+1} /3$ and $d(f,g)=\|f-g\|_2$})  that
\begin{align}
\inf_{\hat{f}\in \hat{\mathcal{F}}}\sup_{f_0\in \mathcal{F}_0}\mathbb{E}_{f_0}\|\hat{f}-f_0\|_2^2
\gtrsim L^2\delta_n 2^{j_n} \Big(1-\frac{I(F;Y)+\log 2}{\log(|\FF_0|/N_{t}^{\max})}\Big),\label{eq: LB:risk2}
\end{align}
where $I(F;Y)$ is the mutual information between the random variables $F$ and $Y$.

%Next let us take take a uniform random variable $M$ on the set $ \mathcal{F}_0$ and consider the Markov chain $M\rightarrow X\rightarrow Y\rightarrow \hat{f}$. Let us denote by $\mathcal{H}(B^{(1)},..,B^{(m)};\mathcal{F})$ the collection of $X=(X^{(1)},...,X^{(m)})$ measurable, binary random vectors $Y=(Y^{(1)},...,Y^{(m)})$, such that $Y^{(i)}$ given $X^{(i)}$ is independent from $Y^{(-i)},X^{(-i)}$ and over the class $\mathcal{F}$ the expected value of bits decoding $Y^{(i)}$ is bounded above by $B^{(i)}$, i.e. for all $f\in\mathcal{F}$ we have $\mathbb{E}_{f}[l(C_{Y^{(i)}})]\leq B^{(i)}$, $i=1,...,m$. Then by applying Lemma 

To lower bound the right-hand side, first
 note that $N_{t}^{\max}=\sum_{i=1}^{{\tilde{t}}}{2^{j_n} \choose i}< 2 {2^{j_n}  \choose {\tilde{t}}}\leq 2(e2^{j_n}/{\tilde{t}})^{\tilde{t}}$ and therefore, for ${\tilde{t}}=2^{j_n-1}/3$ {(i.e. $t^2=L^2\delta_n2^{j_n+1}/3$)},
\begin{align*}
\log(|\FF_0|/N_t^{\max})\geq 2^{j_n} \log (2 (6e)^{-1/6}2^{-2^{-j_n}})\geq 2^{j_n-1}/3.
\end{align*}
Hence, recalling that  $2^{j_n}=\delta_n^{-\frac{1}{1+2s}}$ we see that to prove
\begin{align}
\inf_{\hat{f}\in \hat{\mathcal{F}}}\sup_{f_0\in \mathcal{F}_0}\mathbb{E}_{f_0}\|\hat{f}-f_0\|_2^2\gtrsim \delta_n^{2s/(1+2s)}\label{eq: help:LB}
\end{align}
and as a consequence
to derive the statement of the theorem it is sufficient to 
show that 
%can be reformulated as
%\begin{align*}
%&\inf_{\hat{f}\in \mathcal{F}_{dist}(B^{(1)},...B^{(m)};B_{2,\infty}^s(L))}\sup_{f_0\in B_{2,\infty}^s(L)}\mathbb{E}_{f_0}\|\hat{f}-f_0\|_2^2\nonumber\\
%&\qquad\qquad\gtrsim \inf_{Y\in \mathcal{H}(B^{(1)},..,B^{(m)};B_{2,\infty}^s(L))}\delta_n^{\frac{2s}{1+2s}}\Big( 1-\frac{I(M;Y)+\log 2}{\delta_n^{-1/(1+2s)}/6}\Big).
%\end{align*}
%
%
%
%Hence by noting that $ \mathcal{F}_0\subset B_{2,\infty}^{s}(1)$ it is sufficient to verify the inequality
\begin{align}
I(F;Y)\leq \delta_n^{-1/(1+2s)}/8+O(1)\label{eq: hulp:endproof}.
\end{align}

Observe that for the class of distributed estimators $\hat{\mathcal{F}}=\mathcal{F}_{dist}(B^{(1)},\ldots, B^{(m)};B_{2,\infty}^s(L))$, by definition the following inequality holds
\begin{align}
E^{(i)}l(Y^{(i)})= \frac{1}{|\mathcal{F}_0|}\sum_{f\in\mathcal{F}_0} E_f^{(i)} l(Y^{(i)})\leq B^{(i)},\label{eq: UB:info:length}
\end{align}
where the expectation is taken over the joint distribution of the random variable $F$ and $P_f^{(i)}$, $f\in\mathcal{F}_0$. Next note that for $\delta_n\leq  m/({2^{11}}L^2 n\log n)$ the conditions of Lemma \ref{lem: lem5_redo} are satisfied hence by applying the lemma {(with $\delta^2=L^2\delta_n$ and $d=\delta_n^{-\frac{1}{1+2s}}$)} we get
%(with $\delta=\delta_n^{1/2}$, $d:=2^{j_n}=\delta_n^{-\frac{1}{1+2s}}$, $f=M$, $X=X^{(i)}$, $Y=Y^{(i)}$, $i=1,...,m$), for arbitrary $Y\in \mathcal{H}(B^{(1)},..,B^{(m)}; \mathcal{F}_0)$
\begin{align}
I(F;Y)
&\leq 2L^2n\delta_nm^{-1}\sum_{i=1}^m\min\Big\{ 2^{10}\log(m\delta_n^{-\frac{1}{1+2s}}) H(Y^{(i)}),\delta_n^{-\frac{1}{1+2s}}\Big\}+4\log 2\nonumber\\
&\leq 2L^2n\delta_n m^{-1}\delta_n^{-\frac{1}{1+2s}}\sum_{i=1}^m\Big( 2^{11}\log(n)\delta_n^{\frac{1}{1+2s}} B^{(i)}\wedge 1\Big)+O(1),\label{eq: help:mutual:info}
\end{align}
where the last inequality follows from Lemma \ref{lem: Shannon} and assertion \eqref{eq: UB:info:length}. Since from the definition 
of $\delta_n$ it follows that 
$$\delta_n\leq \frac{2^{-4}L^{-2}m n^{-1}}{  \sum_{i=1}^m\big[2^{11}\log(n)\delta_n^{\frac{1}{1+2s}} B^{(i)}\wedge 1\big] },$$
the right-hand side of \eqref{eq: help:mutual:info} is further bounded by $2^{-3}\delta_n^{-\frac{1}{1+2s}}+O(1)$, finishing the proof of assertion \eqref{eq: hulp:endproof} and concluding the proof of the theorem.

Note that we have used the properties of the distributed estimation class $\hat{\mathcal{F}}$ only in assertion \eqref{eq: UB:info:length}, hence for any distributed method satisfying this inequality we have that 
\begin{align}
\inf_{\hat{f}\in\hat{\mathcal{F}}}\sup_{f_0\in B_{2,\infty}^s(L)} E_{f_0}\|\hat{f}-f_0\|_2^2 \gtrsim \delta_n^{\frac{2s}{1+2s}}.\label{eq: LB:average:bits}
\end{align}

\subsection{Proof of Theorem \ref{theorem: minimaxL2UB}}\label{sec: minimaxL2UB}
First we give the algorithm achieving the upper bound. Let us introduce the notation $\eta =\big( \lfloor(n^{\frac{1}{1+2s}}\log (n)/B)^{{(1+2s)}/({2+2s})}\rfloor\vee 1\big)\wedge m$. Then  we group the local machines into $\eta$ groups 
and let the different groups work on different parts of the signal as follows:
the  machines with indexes $1\leq i\leq  m/\eta$ each transmit the approximations $Y_{jk}^{(i)}$ of the observations $X_{jk}^{(i)}$ for $1\leq 2^{j}+k\leq (B/\log n)\wedge n^{1/(1+2s)}$ using Algorithm \ref{alg: transmit:number}. If $\eta>1$ then the next machines, 
with indexes $m/\eta< i\leq 2 m/\eta$, each transmit the approximations $Y_{jk}^{(i)}$
for $B/\log n< 2^{j}+k\leq 2 B/\log n$,  and so on.  The last machines with numbers 
$(\eta-1) m/\eta< i\leq  m $  transmit $Y_{jk}^{(i)}$ for $(\eta-1) B/\log n< 2^{j}+k\leq \eta B/\log n$. Then in the central machine we average the corresponding transmitted approximated noisy coefficients in the obvious way. Formally, using the notation $\mu_{jk}= \big\lceil(2^{j}+k)\log (n)/ B\big\rceil-1$,  the aggregated estimator $\hat f$ is the function with wavelet 
coefficients given by 
\begin{align*}
\hat{f}_{jk}=
\begin{cases}
mean\{Y_{jk}^{(i)}:\, \frac{\mu_{jk} m}{\eta} < i \leq \frac{(\mu_{jk}+1) m}{\eta}\},& \text{if $2^{j}+k\leq \frac{\eta B}{\log n}$},\\
0, & \text{else}.
\end{cases}
\end{align*}
The procedure is summarized as Algorithm \ref{alg: nonadapt:L2:case2}.

\begin{algorithm}
\caption{Algorithm for the $L_2$-norm}\label{alg: nonadapt:L2:case2}
\begin{algorithmic}[1]
\BState \textbf{In the local machines}:
\For{$\ell= 1$ to $\eta$ }
\For {$i =\lfloor (\ell-1) m/\eta\rfloor +1$ to $\lfloor\ell m/\eta\rfloor$}
\For {$ 2^{j}+k= \lfloor (\ell-1) B/\log n\rfloor+1$ to $\lfloor\ell B/\log n\rfloor$}
\State \text{$Y_{jk}^{(i)}$ :=TransApprox($X_{jk}^{(i)}$)}
\EndFor
\EndFor
\EndFor
\BState \textbf{In the central machine}:
\For {$ 2^j+k=1$ to $\lfloor(\eta B/\log n)\wedge n^{1/(1+2s)}\rfloor$}
\State \text{$\hat{f}_{jk}:=mean\{Y_{jk}^{(i)}:\, \mu_{jk}  m/\eta<  i \leq (\mu_{jk}+1)  m/\eta \}$}
\EndFor
\State Construct: $\hat f = \sum \hat f_{jk}\psi_{jk}$.
\end{algorithmic}
\end{algorithm}

 In the algorithm described above each machine transmits the approximations of  at most $n^{1/(1+2s)}\wedge (B/\log n)$ noisy coefficients. Note that for any $f\in B_{2,\infty}^s(L)$ we have that $f_{jk}^2\leq \sup_{j}2^{js}\sum_k f_{jk}^2\leq L^2$, hence in view of Lemma \ref{lem: approx} (with $|\mu|=|f_{0,jk}|\leq L$) the approximation satisfies 
\begin{align*}
0\leq |X_{jk}^{(i)}-Y_{jk}^{(i)}|1_{\mathcal{E}}\leq 1/\sqrt{n}, \quad |Y_{jk}^{(i)}|\leq \sqrt{n},\quad\text{and}\quad l(Y_{jk}^{(i)})\leq\log n,
\end{align*}
where the set $\mathcal{E}$ was defined in \eqref{def: setE} and satisfies that $P_X(\mathcal{E})\leq e^{-cn}$, for some $c>0$. Therefore  we need at most $B$ bits to transmit
  $ n^{1/(1+2s)}\wedge (B/\log n)$ coefficients, hence $\hat{f}\in\mathcal{F}_{dist}(B,...,B;B_{2,\infty}^s(L))$. 

%Next for convenience we introduce the notation $\eps_{jk}^{(i)}=X_{jk}^{(i)}- Y_{jk}^{(i)}\in[0,n^{-1/2}]$. The estimator $\hat{f}$ is given by its wavelet coefficients $\hat{f}_{jk}$, $j\in\mathbb{N}, k\in K_j$. For $2^j+k>n^{1/(1+2s)}\wedge (B/\log n)$ we have $ \hat{f}_{jk}=0$, 
%while for $2^j+k\leq n^{1/(1+2s)}\wedge (B/\log n)$,
%\begin{align*}
%\hat{f}_{jk}=\frac{1}{m}\sum_{i=1}^m Y_{jk}^{(i)}=\frac{1}{m}\sum_{i=1}^m  f_{0,jk}+\frac{1}{\sqrt{n}}Z_{jk}-\eps_{jk},
%\end{align*}
%where $ \eps_{jk}=m^{-1}\sum_{i=1}^m \eps_{jk}^{(i)}\in[0,n^{-1/2}]$ and $Z_{jk}\stackrel{iid}{\sim}N(0,1)$.

%For convenience we also introduce the notation $j_n=\big\lfloor \log( n^{\frac{1}{1+2s}}\wedge (B/\log n))\big\rfloor$. Then the risk is bounded from above by
%\begin{align}
%\mathbb{E}_{f_0}\|\hat{f}-f_0\|_2^2&\leq \sum_{j\geq j_n}\sum_{k=1}^{2^j}f_{0,jk}^2+\sum_{j=0}^{j_n}\sum_{k=1}^{2^j}\mathbb{E}_{f_0}(\frac{1}{n}Z_{jk}^2+\eps_{jk}^2)\nonumber\\
%&\lesssim  \sum_{j\geq j_n}2^{-2js}\sup_{j\geq j_n} 2^{2js}\sum_{k=1}^{2^j}f_{0,jk}^2+
%\sum_{j=0}^{j_n}\sum_{k=1}^{2^j} n^{-1}\nonumber \\
%&\lesssim 2^{-2j_ns} +2^{j_n}/n\lesssim n^{-2s/(1+2s)}\vee (B/\log n)^{-2s},\label{eq: risk:large:B}
%\end{align}
%where we have used that for $f_0\in B_{2,\infty}^{s}(L)$ we have $|f_{0,jk}|\leq L$ for any $j\geq 1,\, k=1,...,2^j$.

Next for convenience we introduce the notation  $A_{jk}=\{ \lfloor \mu_{jk} m/\eta\rfloor +1,...,\lfloor (\mu_{jk}+1) m/\eta\rfloor \}$ for the collection of machines transmitting the $(j,k)$th coefficient {and note that $\#(A_{jk})\asymp m/\eta$.} Then our aggregated estimator $\hat{f}$ on the set $\mathcal{E}$ satisfies for $2^j+k\leq \eta  B/\log n $ (i.e. the total number of different coefficients transmitted) that 
\begin{align*}
\hat{f}_{jk}&=\frac{1}{\#(A_{jk})}\sum_{i\in A_{jk}}Y_{jk}^{(i)}= f_{0,jk}+{\sqrt{\frac{m}{n \#(A_{jk})}}}Z_{jk}-\eps_{jk},
\end{align*}
where $ \eps_{jk}=\frac{1}{\#(A_{jk}) }\sum_{i\in A_{jk}} \eps_{jk}^{(i)}\in[0,n^{-1/2}]$ and
 $Z_{jk}\stackrel{iid}{\sim}N(0,1)$.

Let  $j_n=\lfloor\log \big(n^{1/(1+2s)} \wedge (\eta B/\log n) \big) \rfloor$.
Then the risk of the aggregated estimator  is bounded as
\begin{align}
\mathbb{E}_{f_0}\| \hat{f}-f_0\|_2^21_{\mathcal{E}}&\leq \sum_{j=j_n}^{\infty}\sum_{k=1}^{2^j} f_{0,jk}^2+
\sum_{j=0}^{j_n}\sum_{k=1}^{2^j}\mathbb{E}_{f_0}({\frac{m}{n \#(A_{jk})}} Z_{jk}^2 +\eps_{jk}^2)1_{\mathcal{E}}\nonumber\\
&\lesssim  \sum_{j=j_n}^{\infty}2^{-2js}\sup_{j\geq j_n}2^{2js}\sum_{k=1}^{2^j} f_{0,jk}^2+\sum_{j=0}^{j_n}\sum_{k=1}^{2^j}{\eta/n} \nonumber\\
&\lesssim  \big(\frac{\eta B}{\log _2n}\wedge n^{1/(1+2s)}\big)^{-2s}+\frac{\eta}{n}\big(\frac{\eta B}{\log _2n}\wedge n^{1/(1+2s)}\big)\nonumber\\
&\asymp 
 \Big\{(\log n)^{\frac{2s}{1+s}}\Big(\frac{n^{1/(1+2s)}}{B\log n}\Big)^{\frac{s}{1+s}}\vee 1 \Big\} n^{-\frac{2s}{1+2s}}\vee \Big( \frac{mB}{\log n}\Big)^{-2s}\nonumber\\
&\lesssim \Big\{(\log n)^{2s}\Big(\frac{n^{1/(1+2s)}}{B\log n}\Big)^{\frac{s}{1+s}}\vee 1 \Big\} n^{-\frac{2s}{1+2s}},\label{UB: alg2}
\end{align}
where we have used that for $f_0\in B_{2,\infty}^{s}(L)$ we have $|f_{0,jk}|\leq L$ for any $j\geq 0,\, k=1,...,2^j$. The above inequality together with 
\begin{align*}
\mathbb{E}_{f_0}\| \hat{f}-f_0\|_2^21_{\mathcal{E}^c}\lesssim n \mathbb{P}_{f_0}(\mathcal{E}^c)\lesssim ne^{-cn}=o(n^{-1}) 
\end{align*}
 concludes the proof of the theorem.

\subsection{Minimax bounds for distributed methods in $L_\infty$-norm}\label{sec: minimaxLinftyLB}

Similarly to the $L_{2}$-case we consider the situation where all communication budgets are the same, i.e. $B^{(1)}=...=B^{(m)}=B$.
\begin{theorem}\label{theorem: minimaxLinftyLB}
Consider   $s,L> 0$,  communication constraint $B^{(1)}=...=B^{(m)}=B> 0$, then
\begin{enumerate}[label=(\roman*b)]
\item
if $B\geq \big(n/(\log n)^{3+4s}\big)^{1/(1+2s)}$, then  \label{(ib)}
\begin{align*}
\inf_{\hat{f}\in\mathcal{F}_{dist}(B,\ldots,B; B_{\infty,\infty}^{s}(L))}\,
\sup_{f_0\in B_{\infty,\infty}^s(L)} \mathbb{E}_{f_0}\|\hat{f}-f_0\|_{\infty} \gtrsim (n/\log n)^{-\frac{s}{1+2s}}.
\end{align*}
\item
if  $(n\log (n)/m^{2+2s})^{{1}/({1+2s})}\leq B <  \big(n/(\log n)^{3+4s}\big)^{1/(1+2s)}$, then \label{(iib)}
\begin{align*}
\inf_{\hat{f}\in\mathcal{F}_{dist}(B,\ldots,B; B_{\infty,\infty}^{s}(L))}\,
\sup_{f_0\in B_{\infty,\infty}^s(L)} \mathbb{E}_{f_0}\|\hat{f}-f_0\|_{\infty} \gtrsim   \big(\frac{n^{\frac{1}{1+2s}}}{B(\log n)^{\frac{3+4s}{1+2s}}}\big)^{\frac{s}{2+2s}} (\frac{n}{\log n})^{-\frac{s}{1+2s}}.
\end{align*}
\item
if $(n\log (n)/m^{2+2s})^{{1}/({1+2s})}> B$\label{(iiib)}, then
\begin{align*}
\inf_{\hat{f}\in\mathcal{F}_{dist}(B,\ldots,B; B_{\infty,\infty}^{s}(L))}\, \sup_{f_0\in B_{\infty,\infty}^s(L)} 
\mathbb{E}_{f_0}\|\hat{f}-f_0\|_{\infty} \gtrsim  \Big(\frac{n\log n}{m}\Big)^{-\frac{s}{1+2s}}.
\end{align*}
\end{enumerate}
\end{theorem}

This theorem is actually a direct consequence of the following more general theorem where the communication thresholds can vary between the machines.

\begin{theorem}\label{thm: minimaxLinftyLBgeneral}
Consider   $s,L> 0$, communication constraints $B^{(1)}, \ldots, B^{(m)} > 0$ and
let the sequence $\delta_n=o(1)$ be defined as the solution to the equation
\eqref{def: delta_n}.
Then in the distributed Gaussian white noise model \eqref{model: GWN} we have that
\begin{align*}
\inf_{\hat{f}\in\mathcal{F}_{dist}(B^{(1)},\ldots,B^{(m)}; B_{\infty,\infty}^s(L))}
\, \sup_{f_0\in  B_{\infty,\infty}^s(L)} \mathbb{E}_{f_0}\|\hat{f}-f_0\|_{\infty} \gtrsim \Big(\frac{n}{\log n}\Big)^{-\frac{s}{1+2s}} \vee  \delta_n^{\frac{s}{1+2s}}.
\end{align*}
\end{theorem}

\begin{proof}
First of all we note that in the non-distributed case where all the information is available in the global machine the minimax $L_{\infty}$-risk is $(n/\log n)^{-\frac{s}{1+2s}}$. Since the class of distributed estimators is clearly a {subset of the class of all} estimators this will be also a lower bound for the distributed case. The rest of the proof goes similarly to the proof of Theorem \ref{sec: minimaxL2LB}.

 First we construct a finite subset $\FF_0\subset B_{\infty,\infty}^s(L)$ and then give a lower bound for the minimax risk over it. Let us denote by ${K}_j$ the largest set of Daubechies wavelets at resolution level $j$ with disjoint supports. Note that $|{K}_j|\geq c_0 2^j$ (for large enough $j$ and sufficiently small $c_0>0$). Let us again multiply $\delta_n$ with a sufficiently small constant and work with this $\delta_n$ in the rest of the proof
\begin{align}
\delta_n:={c_02^{-13}}L^{-2}\min\Big\{\frac{m}{n \log n }, \frac{m}{n \sum_{i=1}^m [\delta_n^{\frac{1}{1+2s}}\log (n) B^{(i)} \wedge  1]} \Big\}.\label{def: delta_n_new2}
\end{align}

 Let $j_n=\lfloor(\log \delta_n^{-1})/(1+2s)\rfloor$ and 
for $\beta\in\{-1,1\}^{|{K}_{j_n}|}$ let $f_\beta\in L_\infty[0,1]$ be the function with wavelet coefficients
 \begin{align*}
 f_{\beta, jk}=
 \begin{cases}
L\delta_n^{1/2}\beta_{k}, & \text{if}\quad j=j_n,\, k\in  {K}_{j_n},\\
 0, & \text{else}.
 \end{cases}
\end{align*}
Now let 
 $\FF_0=\{f_\beta:\, \beta_{k}\in\{-1,1\},k\in {K}_{j_n}\}$.
% and consider the a maximal subset $ \mathcal{F}_0\subset$ satisfying that for all $\beta,\beta'\in \mathcal{F}_0$ the inequality $\|\beta-\beta'\|_2\geq |{K}_{j_n}|/8$ holds.

Note that each function $f_\beta\in\FF_0$ belongs to the set $B_{\infty,\infty}^s(L)$, since
\begin{align*}
\|f_\beta\|_{B_{\infty,\infty}^s}=\sup_{j,k}2^{(s+1/2)j}f_{\beta, jk}^2=2^{(s+1/2)j_n}\sup_{k\in {K}_{j_n}}L\delta_n^{1/2}=L2^{(s+1/2)j_n}\delta_n^{1/2}\leq  L.
\end{align*} 
Furthermore, if $f_\beta \not = f_{\beta'}$, {then there} exists a $k'\in {K}_{j_n}$ such that $\beta_{k'}\neq \beta'_{k'}$. Then due to the disjoint support of the corresponding 
Daubechies' wavelets $\psi_{j_n,k}$, $k\in {K}_{j_n}$ the $L_{\infty}$-distance between 
the two functions  is bounded from below by
\begin{align*}
\|f_\beta-f_{\beta'}\|_{\infty}\geq  |f_{j_nk'}-f'_{j_nk'}|\cdot \|\psi_{j_n,k'}\|_{\infty}
{\gtrsim} 2^{j_n/2+1}\delta_n^{1/2}\geq \delta_n^{\frac{s}{1+2s}}.
\end{align*}

Next observe that for an arbitrary set of estimators $\hat{\mathcal{F}}$
\begin{align*}
\inf_{\hat{f}\in\hat{\mathcal{F}}}\sup_{f_0\in B_{\infty,\infty}^{{s}}(L)}\mathbb{E}_{f_0}\|\hat{f}-f_0\|_\infty\geq \inf_{\hat{f}\in \hat{\mathcal{F}}}\sup_{f_0\in \mathcal{F}_0}\mathbb{E}_{f_0}\|\hat{f}-f_0\|_\infty.
\end{align*}

Now let $F$ be a uniform random variable on the set $\FF_0$. Then in view of Fano's inequality 
(see Theorem \ref{lem: Fano:Wainwright} with $t=\delta_n^{s/(1+2s)}$ and $p=1$) we get that
\begin{align*}
\inf_{\hat{f}\in\hat{\mathcal{F}}}\sup_{f_0\in  \mathcal{F}_0} \mathbb{E}_{f_0}\|\hat{f}-f_0\|_{\infty} 
\gtrsim
\delta_n^{\frac{s}{1+2s}}\Big(1-\frac{I(F;Y)+\log 2}{\log |\FF_0|} \Big).
\end{align*}
Hence, since $\log |\FF_0|\geq |{K}_{j_n}|\geq c_0 2^{j_n}=c_0 \delta_n^{-1/(1+2s)}$,  it remains to show that $I(F;Y)\leq (c_0/2)\delta_n^{-1/(1+2s)}+O(1)$.

In view of Lemma \ref{lem: lem5_redo} (applied with $\delta=\delta_n^{1/2}$, $d=|{K}_{j_n}|=c_0\delta_n^{-\frac{1}{1+2s}}$, $X=X^{(i)}$, $Y=Y^{(i)}$, $i=1,...,m$, and noting that $\delta_n\leq  m/({2^{11}}L^2n\log n)$ hence the conditions are fulfilled)
\begin{align*}
I(F;Y)
&\leq 2L^2 n\delta_n m^{-1}\delta_n^{-\frac{1}{1+2s}}\sum_{i=1}^m\Big( 2^{10}\log(n)\delta_n^{\frac{1}{1+2s}}H(Y^{(i)}) \wedge c_0\Big)+4\log 2,\\
&\leq 2^{12}L^2n\delta_n m^{-1}\delta_n^{-\frac{1}{1+2s}}\sum_{i=1}^m\Big( \log(n)\delta_n^{\frac{1}{1+2s}} B^{(i)}\wedge 1\Big)+O(1)\\
&\leq (c_0/2) \delta_n^{-\frac{1}{1+2s}}+O(1),
\end{align*}
where the second inequality follows from Theorem \ref{lem: Shannon} and assertion \eqref{eq: UB:info:length} for $\hat{\mathcal{F}}=\mathcal{F}_{dist}(B^{(1)},\ldots, B^{(m)};B_{\infty,\infty}^s(L))$ and the third by the definition of $\delta_n$, see \eqref{def: delta_n_new2}. Hence we can conclude that
\begin{align}
\inf_{\hat{f}\in \mathcal{F}_{dist}(B^{(1)},\ldots, B^{(m)};\mathcal{F}_0)}\sup_{f_0\in \mathcal{F}_0}\mathbb{E}_{f_0}\|\hat{f}-f_0\|_\infty\gtrsim \delta_n^{\frac{s}{1+2s}}.\label{eq: help:LB_Linfty}
\end{align}

Note that we have used the properties of the distributed estimation class $\hat{\mathcal{F}}$ only in assertion \eqref{eq: UB:info:length}, hence for any class distributed estimator $\hat{\mathcal{F}}$ satisfying this inequality we have that
\begin{align}
\inf_{\hat{f}\in\hat{\mathcal{F}}}\sup_{f_0\in B_{\infty,\infty}^s(L)} E_{f_0}\|\hat{f}-f_0\|_\infty \gtrsim \delta_n^{\frac{s}{1+2s}}.\label{eq: LB:average:bits:infty}
\end{align}
\end{proof}

Next we give an algorithm providing matching upper bounds in the first two cases. Note that the last case, similarly to the $L_2$-norm is less relevant as using the data available only on a single machine would provide at least as good an estimator as any distributed algorithm. The algorithm is very similar to the $L_2$-case, i.e. Algorithm \ref{alg: nonadapt:L2:case2}, and is basically the rewrite of Algorithm 4 of \cite{szabo:zanten:2018} tailored to the Gaussian white noise model. Here we just highlight the differences compared to Algorithm \ref{alg: nonadapt:L2:case2}. We divide the machines into $\eta = (\lfloor\big(L^2n(\log_2 n)^{2s}/B^{1+2s}\big)^{\frac{1}{2+2s}}\rfloor \wedge m)\vee 1$ equal sized groups ($\eta=1$ corresponds to case (ib), while $\eta>1$ corresponds to case (iib)). Similarly to before machines with indexes $1\leq i\leq  m/\eta$ transmit the approximations $Y_{jk}^{(i)}$ for $1\leq 2^{j}+k\leq \lfloor B/\log_2 n\rfloor\wedge (n/\log_2 n)^{\frac{1}{1+2s}}$, and so on, the last machines with numbers 
$(\eta-1) m/\eta< i\leq  m $  transmit the approximations $Y_{jk}^{(i)}$ for $\big((\eta-1)\lfloor B/\log_2 n\rfloor\big)\wedge (n/\log_2 n)^{\frac{1}{1+2s}}< 2^{j}+k\leq \big(\eta\lfloor B/\log_2 n\rfloor\big)\wedge (n/\log_2 n)^{\frac{1}{1+2s}}$. Then in the central machine we average the corresponding transmitted coefficients in the obvious way, similarly to the $L_2$-norm case.
The procedure is summarized as Algorithm \ref{alg: nonadapt:Linfty} and the (up to a logarithmic factor) optimal behaviour is given in 
Theorem \ref{theorem: minimaxLinftyUB} below.

\begin{algorithm}
\caption{Nonadaptive $L_\infty$-method, combined}\label{alg: nonadapt:Linfty}
\begin{algorithmic}[1]
\BState \textbf{In the local machines}:
\For{$\ell= 1$ to $\eta$ }
\For {$i =\lfloor (\ell-1) m/\eta\rfloor +1$ to $\lfloor\ell m/\eta\rfloor$}
\For {$ 2^{j}+k= (\ell-1)\lfloor  B/\log_2 n\rfloor+1$ to $\ell \lfloor B/\log_2 n\rfloor$}
\State \text{$Y_{jk}^{(i)}$ :=TransApprox($X_{jk}^{(i)}$).}
\EndFor
\EndFor
\EndFor
\BState \textbf{In the central machine}:
\For {$ 2^j+k=1$ to $\eta\lfloor B/\log_2 n\rfloor$}
\State \text{$\hat{f}_{jk}:=mean\{Y_{jk}^{(i)}:\, \mu_{jk}  m/\eta<  i \leq (\mu_{jk}+1)  m/\eta \}$}.
\EndFor
\State Construct: $\hat f = \sum \hat f_{jk}\psi_{jk}$.
\end{algorithmic}
\end{algorithm}

\begin{theorem}\label{theorem: minimaxLinftyUB}
Let $s,L> 0$, then the distributed estimator $\hat f$ described 
 in Algorithm \ref{alg: nonadapt:Linfty} belongs to 
 $\mathcal{F}_{dist}(B,\ldots,B; B_{\infty,\infty}^{s}(L))$  and satisfies
\begin{itemize}
\item for $B\geq n^{1/(1+2s)}(\log_2 n)^{2s/(1+2s)}$,
\begin{align*}
\sup_{f_0\in B_{\infty,\infty}^{s}(L)}\mathbb{E}_{f_0}\| \hat{f}-f_0\|_\infty\lesssim (n/\log_2 n)^{-\frac{s}{1+2s}};
\end{align*}
\item for $\big(n(\log_2 n)/m^{2+2s}\big)^{{1}/({1+2s})} \vee \log_2 n\leq B< n^{1/(1+2s)}(\log_2 n)^{2s/(1+2s)}$,
\begin{align*}
\sup_{f_0\in B_{\infty,\infty}^{s}(L)}\mathbb{E}_{f_0}\| \hat{f}-f_0\|_\infty\lesssim  M_n\Big(\frac{n^{\frac{1}{1+2s}} }{B(\log_2 n)^{\frac{3+4s}{1+2s}}}\Big)^{\frac{s}{2+2s}}(n/\log_2 n)^{-\frac{s}{1+2s}},
\end{align*}
with $M_n=(\log_2 n)^{s\vee \frac{3s}{2+2s}}$.
\end{itemize}
\end{theorem}

The proof of the theorem follows the same reasoning as the proof of Theorem \ref{sec: minimaxL2UB} but for the $L_{\infty}$-norm and it is basically follows from the proof of Theorem 2.8 of  \cite{szabo:zanten:2018} tailored to the Gaussian white noise model.

\section{Definitions and notations for wavelets}\label{sec: wavelets}
In this section we collect some notations and definitions about wavelets, a more detailed description can be found for instance in \cite{hardle:2012, gine:nickl:2016}.

We consider the Cohen, Daubechies and Vial construction of compactly supported, orthonormal, $N$-regular wavelet basis of $L_2[0,1]$, see for instance \cite{cohen:1993} and let the us use the notation $\{\psi_{jk}:\, j=0,1,..,\, k=1,...,2^{j}\}$. For arbitrary function $f\in L_2[0,1]$ we can consider the wavelet representation
\begin{align*}
f=\sum_{j=0}^{\infty}\sum_{k=1}^{2^{j}}f_{jk}\psi_{jk},
\end{align*}
with $f_{jk}=\langle f,\psi_{jk}\rangle$. Following from the orthonormality of the wavelet basis we have that
\begin{align*}
\|f\|_2^2=\sum_{j=0}^{\infty}\sum_{k=1}^{2^{j}}f_{jk}^2.
\end{align*}

In our analysis we work with the Besov spaces $B_{2,\infty}^s$ and $B_{\infty,\infty}^s$. The corresponding Besov norms for $s\in(0,N)$ are defined as
\begin{align*}
\|f\|_{B_{2,\infty}^s}^2=\sup_{j\geq j_0} 2^{2js}\sum_{k=0}^{2^j-1}f_{jk}^2\quad\text{and}\quad \|f\|_{B_{\infty,\infty}^s}= \sup_{j\geq0, k}\{ 2^{j(s+1/2)}|f_{jk}|\}.
\end{align*}
 Then the Besov spaces $B_{2,\infty}^s, B_{\infty,\infty}^s$ and the corresponding Besov balls $B_{2,\infty}^s(L), B_{\infty,\infty}^s(L)$ of radius $L>0$ are defined as
\begin{align*}
&B_{2,\infty}^s=\{f\in L_2[0,1]:\, \|f\|_{B_{2,\infty}^s}<\infty \},\\
&B_{2,\infty}^s(L)=\{f\in L_2[0,1]:\, \|f\|_{B_{2,\infty}^s}<L \},\\
&B_{\infty,\infty}^s=\{f\in L_2[0,1]:\, \|f\|_{B_{\infty,\infty}^s}<\infty \}\quad\text{and}\\
&B_{\infty,\infty}^s(L)=\{f\in L_2[0,1]:\, \|f\|_{B_{\infty,\infty}^s}<L \},
\end{align*}
respectively. We note that the Besov space $B_{2,\infty}^s$ is larger than the standard Sobolev space where instead of the supremum one would take the sum over the resolution levels $j$.  For $s\neq N$ $B_{\infty,\infty}^s$ is equivalent to the classical H\"older space with regularity $s$, while for integer $s$ they are equivalent to the so called Zygmond spaces, see \cite{cohen:1993}.

\bibliographystyle{acm}
\bibliography{references}

\end{document}